\documentclass[12pt]{amsart}
\usepackage{mathrsfs}
\usepackage{geometry}
\usepackage[all]{xy}

\usepackage{amsmath, amsthm}
\usepackage{amssymb}
\usepackage{amscd}
\usepackage{latexsym}
\usepackage{epsfig}
\usepackage{graphics}
\usepackage{amsfonts}
\usepackage{psfrag}
\usepackage{graphicx}
\usepackage{fancyhdr}

\def\N {{\mathbb{N}}}

\def\R {{\mathbb{R}}}
\def\C {{\mathbb{C}}}

\def\F {{\mathscr{F}}}
\def\G {{\mathscr{G}}}
\def\H {{\mathscr{H}}}
\def\O {{\mathscr{O}}}
\def \T{{\mathbb{T}}}
\def\v{{\bf{v}}}
\def\u{{\bf{u}}}
\def\b{{\bf{b}}}
\def\0{{\bf{0}}}
\def\w{{\bf{w}}}
\def \al{{\alpha}}

\def \Ga{{\Gamma}}
\def\cdots {{\cdot\cdot\cdot}}
\def\ra {{\rightarrow}}

\newcommand{\ds}{\displaystyle}

\usepackage{mathabx,epsfig}
\def\acts{\mathrel{\reflectbox{$\righttoleftarrow$}}}

\newtheorem{theorem}{Theorem}[section]
\newtheorem{question}[theorem]{Question}
\newtheorem{lemma}[theorem]{Lemma}
\newtheorem{proposition}[theorem]{Proposition}
\newtheorem{corollary}[theorem]{Corollary}

\theoremstyle{definition}
\newtheorem{definition}[theorem]{Definition}
\newtheorem{example}[theorem]{Example}
\newtheorem{remark}[theorem]{Remark}
\newtheorem{remark on notation}[theorem]{Remark on Notation}
\newtheorem{notation}[theorem]{Notation}
\newtheorem{setup}[theorem]{Setup}

\numberwithin{equation}{section}

\begin{document}

\title[Graph Cohomology and Betti Numbers]{on graph cohomology and Betti numbers of Hamiltonian GKM manifolds}

\author{Shisen Luo}

\address{Department of Mathematics, Cornell University,
Ithaca, NY 14853-4201, USA}

\email{{\tt ssluo@math.cornell.edu}}

\subjclass[2010]{53D05, 97K30} \keywords{graph cohomology, Betti numbers}

\begin{abstract}
In this paper we introduce the concept of characteristic number that are proven to be useful in the study of the combinatorics of graph 
cohomology. We claim that it is a good combinatorial counterpart for geometric Betti numbers. We then use this concept and tools built along the way to study Hamiltonian GKM manifolds whose moment maps are in general 
position. We prove some connectivity properties of the their GKM graphs and show an upper bound of their second Betti numbers,
which allows us to conclude that these manifolds, in the case of dimension $8$ and $10$, have non-decreasing even Betti numbers up to half dimension. 
\end{abstract}

\date{\today}
\maketitle \tableofcontents
\section{\bf Introduction}\label{sec: introduction}
 In their landmark paper \cite{GKM}, Goresky, Kottwitz and MacPherson showed that in certain circumstances, the computation of equivariant cohomology, a topological problem, can be converted into a combinatorial one.  More concretely, assume a torus $\T^l, l\geq 2$, acts on a smooth manifold $M$. Under certain assumptions, which we will assume by saying the action is GKM, we can assign to $M$ a simple graph (undirected, no loops, no multiple edges) $\Ga=(V,E)$, which we will call the GKM graph of $M$,  and a map $\alpha: E\rightarrow \C[x_1, x_2, ..., x_l]_1$, where $\C[x_1, x_2, ..., x_l]_1$ denotes the set of non-zero linear polynomials in $x_1, x_2, ..., x_l$.  The GKM theorem establishes the following isomorphism
 \begin{equation}\label{eq:GKM}
 H_{\T^l}^{*}(M; \C)\cong \left\{(f_1, f_2, ..., f_{|V|})\in \bigoplus_{i=1}^{|V|}\C[x_1, x_2, ..., x_l]\bigg| \alpha(e_{ij})\big| f_i-f_j, \forall e_{ij}\in E\right\}.
 \end{equation}
The right hand side of the isomorphism will be denoted by $H^{*}_{\T^l}(\Ga,\alpha)$, or $H^{*}_{\T^l}(\Ga)$ when there is no chance of confusion. It is called the (equivariant) graph cohomology of $\Ga$.  The graph cohomology $H^{*}_{\T^l}(\Ga)$ is naturally graded. We assign degree $1$ to each variable $x_{i}$, then the isomorphism \eqref{eq:GKM} divides the degrees in half (the left hand side only has even degree elements).
  
  Among the many works inspired by \cite{GKM}, Gullemin and Zara \cite{GZ:graph} studied the combinatorial properties of $H_{\T^l}^{*}(\Ga)$, 
  and showed that many familiar theorems in geometry, such as the ABBV localization theorem, are in fact theorems about graphs. In other words, they proved some theorems in graph theory inspired by the well-known geometric facts. In this paper, we consider the problem in the opposite direction: Can we find new properties of GKM manifolds, which have no known geometric proof, by studying the properties of the graphs?

 We are interested in the case when $M$ is a symplectic manifold and $\T^l\acts M$ is a Hamiltonian action (we refer the reader to \cite{CdS:book} for definitions of basic notions in symplectic geometry). We will call $M$ a {\it Hamiltonian GKM manifold} if it is a symplectic manifold with a Hamiltonian torus action that is GKM. The manifold is always assumed to be compact. In this case, there exists a map 
  $$\phi: V\rightarrow \R^l,$$
 called a {\it moment map}, and the map $\alpha: E\rightarrow \C[x_1, ..., x_l]_1$ can be {\it induced} from $\phi$: if $\phi(v_i)=(y_1, ... ,y_l)$, 
  $\phi(v_j)=(z_1, ..., z_l)$, and $e_{ij}\in E$,  then $\alpha(e_{ij})=(z_1-y_1)x_1+(z_2-y_2)x_2+\cdots +(z_l-y_l)x_l$.   There will be assumptions on the map $\phi$ which ensures that $\alpha(e_{ij})$ is non-zero. This definition leaves room for some ambiguity since the edges are not oriented: $e_{ij}$ and $e_{ji}$
  are the same. Our definition of $\alpha(e_{ij})$ is only defined up to sign, but we shall see this ambiguity will not cause any problem.
  
  For any GKM action $\T^l \acts M$ with $l\geq 2$, we can always restrict it to a smaller torus action $\T^2\acts M$ that is still GKM.  So we assume the following setup for the remainder of the paper.
  \begin{setup}\label{setup1}
  Let $\Ga=(V, E)$ be a simple graph, and $\phi: V\ra \R^2$ a map in {\it general position} (we will make this precise in Definition~\ref{def:general}).  For simplicity, we will always use $m$ to stand for $|V|$. Let $\alpha: E\rightarrow \C[x,y]_1$ be induced by $\phi$, and
  \[H_{\T}^{*}(\Ga)=\left\{(f_1, f_2, ..., f_{m})\in \bigoplus_{i=1}^{m}\C[x,y]\bigg| \alpha(e_{ij})\big| f_i-f_j, \forall e_{ij}\in E\right\}.\]

  \end{setup}
  \begin{definition}\label{def:general}
  We say $\phi: V\ra \R^2$ is in {\it general position} if no three points in $\phi(V)$ lie on the same line. In particular, this implies $\phi$ is injective.
  \end{definition}
  
  With this setup, we show in Section~\ref{sec: freeness} that $H_{\T}^{*}(\Ga)$ is always a free module over $\C[x,y]$ of dimension $m$. For any set of homogeneous generators of $H_{\T}^{*}(\Ga)$ as module over $\C[x,y]$,
  \[\gamma_{1}, \gamma_{2}, ..., \gamma_{m},\]
  we set
  $c_{i}(\Gamma)= \big|\{j\in \N\big|\ \mbox{degree}(\gamma_{j})=i\}\big|$, the number of degree $i$ generators.
  
  Although the set of generators is not unique, it is a fact for graded free modules that the number $c_{i}$ is independent of the choice of the set of generators, hence well-defined. 
  
  \begin{definition}\label{def:characteristic}
  We call $c_{i}(\Ga)$ the $i$-th characteristic number of $\Ga$. 
  \end{definition}
  
  These $c_{i}$'s will serve as the combinatorial counterpart for the geometric Betti numbers of the manifolds in this paper. This is different from the combinatorial Betti numbers defined in \cite{GZ:graph} as follows, adapted to our setup.
    
  Pick $\xi \in \R^2$ a {\it generic direction}, which assures that $\phi(v_i)\cdot\xi\neq \phi(v_j)\cdot\xi$ for $i\neq j$. For any $v_{i}\in V$, define $$\sigma(v_{i})=\bigg|\{v_{j}\in V\big|e_{ij}\in E, \phi(v_{j})\cdot \xi < \phi(v_{i})\cdot \xi \}\bigg|,$$ called the {\it index} of $v_{i}$.
  Define $\beta_{k}(\Ga)$ to be the number of vertices of index $k$, the {\it $k$-th (combinatorial) Betti number of $\Ga$}. Guillemin and Zara showed in \cite{GZ:graph} that with the existence of axial function, although the indices depend on the choice of $\xi$, the combinatorial Betti numbers do not. 
  
  We note that when $\Ga$ is actually the GKM graph of a manifold, then these three notions: characteristic numbers, combinatorial Betti numbers and geometric Betti numbers all agree. 
  
  These $\beta_{i}$'s were used as the combinatorial counterpart of geometric Betti numbers in \cite{GZ:graph}. They do have some nice properties that resemble the geometric ones. For example, Poincare duality holds. 
  But there are certain shortfalls. First of all, they are well-defined only for very restrictive graphs: regular graphs with axial functions. This makes inductive arguments difficult. Second of all, for a connected regular graph equipped with an axial function, one would expect the zeroth Betti number to be $1$, but this is not the case, as the following simple example shows.
  \begin{figure}[!ht]
   \includegraphics[height=60mm]{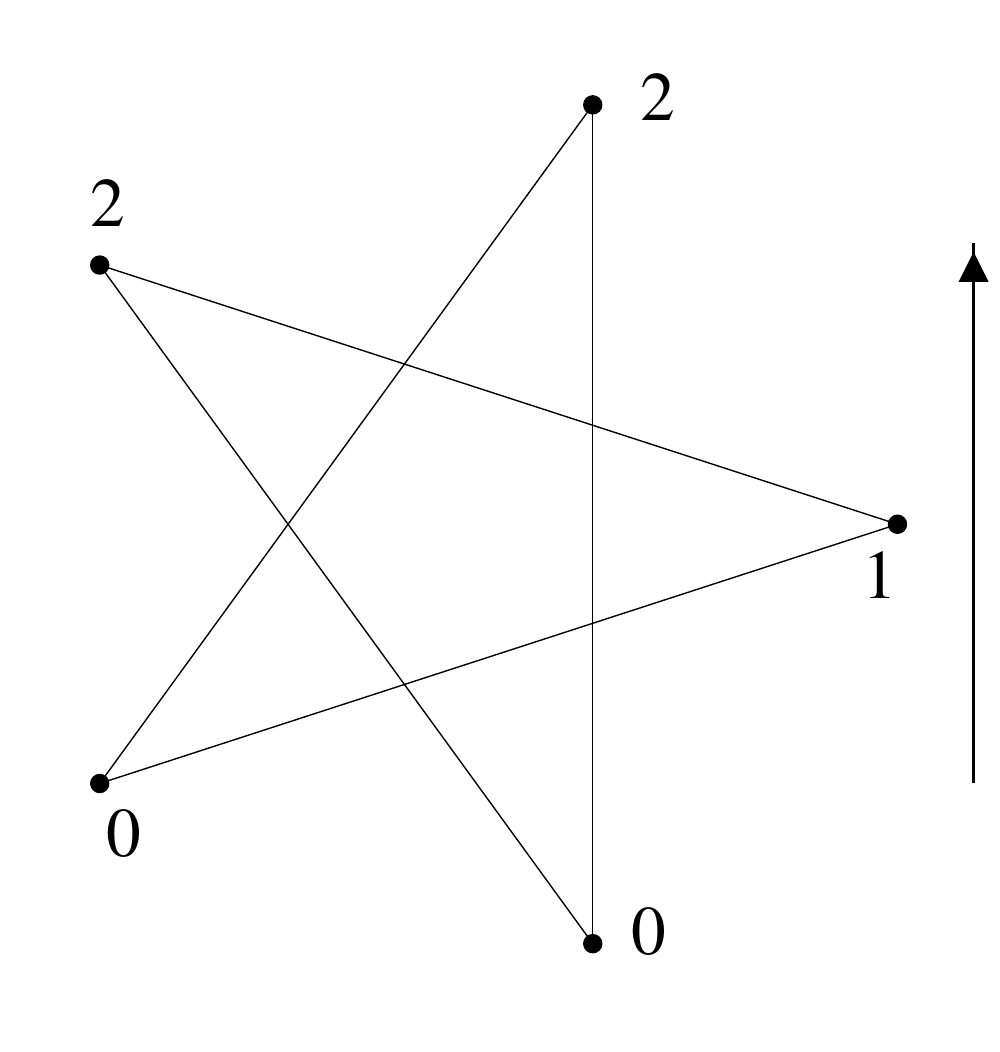}
    \caption{A regular graph with axial function but $\beta_{0}=2$.}
\label{fig:ZerothBetti2}
  \end{figure}
  \begin{example}
  As shown in Figure~\ref{fig:ZerothBetti2}, the regular two-valent graph has an axial function induced by the given embedding. The arrow points in the $\xi$ direction and the number beside a vertex indicates the index of the that vertex. We can see although Poincare duality still holds, the graph has an undesirable $\beta_{0}$, which is $2$.  
  \end{example}
   
   In contrast, $c_{i}$'s are defined in a much more general setting and $c_{0}$ equals to the number of connected components of the graph. These properties provide the first motivation for studying these numbers.
   In Section~\ref{sec: characteristic}, we study the basic properties of characteristic numbers. We will see how they resemble, or fail to resemble, the geometric Betti numbers in various senses.
   In the remaining two sections, we study the properties of actual GKM graphs, i.e., graphs that are in fact the GKM graphs of compact Hamiltonian GKM manifolds. 
   In Section~\ref{sec: connectivity}, using our tools of characteristic numbers and the fact that top Betti number of a compact oriented manifold is $1$, we deduce some connectivity properties of (actual) GKM graphs, under the assumption that the moment map is in general position. The main theorems are Theorem~\ref{theorem:edge connected} and Theorem~\ref{theorem:vertex connected}.
   
   Section~\ref{sec: bound} is largely motivated by the following question:
   \begin{question}
   Suppose that a symplectic manifold $(M, \omega)$ admits a Hamiltonian $S^1$ action with isolated fixed points. Does $(M, \omega)$ satisfy the hard Lefschetz property?
   \end{question}
   Yael Karshon brought up the question at a workshop at BIRS in 2005 and is listed as Problem 4.2 in \cite{Banff}. The problem was said to have existed for over ten years before then, but the origin of it is unclear.  Susan Tolman pointed out an easier version of the question. 
   \begin{question}[Problem 4.3 in \cite{Banff}]\label{question:tolman}
   Suppose that a symplectic manifold $(M, \omega)$ of dimension $2d$ admits a Hamiltonian $S^1$ action with isolated fixed points. Are the Betti numbers of $M$ unimodal? That is, do they satisfy 
   \[\beta_1\leq \beta_3\leq \cdots \leq \beta_{2\lceil\frac{d}{2}\rceil-1}\]
   and 
   \[\beta_0\leq \beta_2\leq \cdots \leq \beta_{2\lfloor\frac{d}{2}\rfloor},\]
  where $2\lceil\frac{d}{2}\rceil-1$ and $2\lfloor\frac{d}{2}\rfloor$ are  respectively the largest odd integer and even integer no greater than $d$?
   \end{question}
    
    A very special case of manifolds which have Hamiltonian $S^1$ actions with isolated fixed points are Hamiltonian GKM manifolds. As we have already stated, these manifolds have the virtue that their equivarant cohomology, and hence ordinary cohomology, can be computed combinatorially as graph cohomology. Moreover, all the odd Betti numbers of these manifolds vanish, so we only need to worry about the even Betti numbers when considering Tolman's question. 
    
    Using the tools and results we will develop in Section~\ref{sec: characteristic} and ~\ref{sec: connectivity}, we prove in Section~\ref{sec: bound} an upper bound for $\beta_2$ of Hamiltonian GKM manifold whose moment map is in general position. This is the content of Theorem~\ref{thm:bound}.  In dimension $8$ and $10$, this will imply that the Betti numbers of these manifolds are unimodal. So our work provides a class of examples which suggests a framework for a positive answer to Karshon's and Tolman's questions.

   {\bf Acknowlegement:} I would like to thank Tara Holm, Bob Connelly, Victor Guillemin, Allen Knutson, Tomoo Matsumura, Edward Swartz and Catalin Zara for many helpful discussions.

\section{\bf Freeness of graph cohomology in dimension two}\label{sec: freeness}
The main theorem in this section, Theorem~\ref{theorem:freeness}, holds in a much more general setting than we assumed for the rest of the paper. In particular, it does not require the existence of $\phi$ and there is no constraints on the map $\alpha$ at all.
\begin{theorem}\label{theorem:freeness}
Given $\Ga=(V,E)$ a simple graph and $\al: E\ra \C[x,y]_1$, then $H_{\T}^{*}(\Ga)\cong (\C[x,y])^m$ as modules over $\C[x,y]$, where $m=|V|$. 
\end{theorem}
\begin{proof}
Let $M=H_{\T}^{*}(\Gamma)$. As a graded module over
$\C[x,y]$, it defines a quasicoherent sheaf over $\mbox{Proj}
\C[x,y]=\C P^1$. We denote this sheaf by $\mathscr{F}$. As a module
over $\C[x,y]$, $M$ also defines a quasicoherent sheaf over
$\mbox{Spec}\C[x,y]=\C^2$. We denote this sheaf by $\mathscr{G}$.
The restriction of $\G$ to $\C^2\backslash\{0\}$ will be denoted by
$\H$, which is a quasicoherent sheaf over $\C^2\backslash\{0\}$.

\begin{lemma}\label{lemma: locally free over P1}
The sheaf $\F$ is locally free.
\end{lemma}
\begin{proof}[Proof of Lemma~\ref{lemma: locally free over P1}]
$\C P^1$ is covered by $D(x)=\mbox{Spec}\C[\frac{y}{x}]$ and
$D(y)=\mbox{Spec}\C[\frac{x}{y}]$. Now that $\F(D(x))=(M_{x})_{0}$,
the degree $0$ part of the localized module $M_x$, it follows from
the definition of $M$ that $\F(D(x))$ is a torsion free
$\C[\frac{y}{x}]$-module, hence a free $\C[\frac{y}{x}]$-module.
Similarly $\F(D(y))$ is a free $\C[\frac{x}{y}]$-module.
\end{proof}

\begin{lemma}\label{lemma:pullback of sheaf}
Let $\pi: \C^2 \backslash\{0\}\rightarrow \C P^1$ be the natural
projection, then $\pi^{*}\F=\H$.
\end{lemma}
\begin{proof}[Proof of Lemma~\ref{lemma:pullback of sheaf}]This can be proved by looking at the distinguished open
subsets of both spaces. The space $\C^2\backslash\{0\}$ is covered by $\mbox{Spec}\C[x,\frac{1}{x},y]$ and $\mbox{Spec}\C[x,y,\frac{1}{y}]$, while $\C P^1$ is covered by $\mbox{Spec}\C[\frac{y}{x}]$ and $\mbox{Spec}\C[\frac{x}{y}]$.

We have the following natural isomorphism
$$(M_{x})_{0}\otimes_{\C[\frac{y}{x}]}\C[x,\frac{1}{x},y]=(M_{x})_{0}
\otimes_{\C[\frac{y}{x}]}(\C[\frac{y}{x}]\otimes_{\C}\C[x,\frac{1}{x}])
=(M_{x})_{0}\otimes_{\C}\C[x,\frac{1}{x}]=M_{x}.$$ Similarly
$$(M_{y})_{0}\otimes_{\C[\frac{x}{y}]}\C[x,y,\frac{1}{y}]=M_{y}.$$
These say exactly $\pi^{*}\F=\H$.
\end{proof}

As an immediate consequence, $\H$ is also locally free. By a theorem
of Grothendieck \cite{Grothendieck}, $\F$ splits as a direct sum of
line bundles. So $\H$ is also a direct sum of line bundles, but the
Picard group of $\C^2\backslash\{0\}$ is trivial, so $\H$ must be a
trivial vector bundle. In particular, $\H(\C^2\backslash\{0\})$ must be a
free module over
$\O_{\C^2\backslash\{0\}}(\C^2\backslash\{0\})=\C[x,y]$.

Now $\G(\C^2\backslash\{0\})=\H(\C^2\backslash\{0\})$ and
$\G(\C^2)=M$, the following lemma will enable us to conclude that
$M$ is a free $\C[x,y]$-module.

\begin{lemma}\label{lemma:iso}
The restriction map $r: \G(\C^2)\rightarrow \G(\C^2\backslash\{0\})$
is an isomorphism.
\end{lemma}
\begin{proof}[Proof of Lemma~\ref{lemma:iso}]
We can think of $\G(\C^2\backslash\{0\})$ as $\bigcap_{ab\neq
0}M_{ax+by}$. The intersection makes sense since $M$ is torsion-free
by definition and thus $M_{ax+by}$ can be thought of as a subset of
$M_{(0)}=\C(x,y)^m$, the localization of $M$ at the zero prime ideal.

Assume $(f_1,...,f_m)\in \G(\C^2\backslash\{0\})=\bigcap_{ab\neq
0}M_{ax+by}\subseteq \C(x,y)^m$.
Then we immediately see that $f_i\in \C[x,y]$ for all $i$. Since the
graph is finite, we can pick a non-zero linear polynomial $px+qy$, such that it
is not a multiple of any $\alpha(e_{ij})$. Then $(f_1,...,f_m)\in
M_{px+qy}$ says that there exists $n$, such that 
\[\alpha(e_{ij})|(px+qy)^n(f_{i}-f_{j})\]
for all $e_{ij}\in E$. So
\[\alpha(e_{ij})|(f_{i}-f_{j}),\]
which says exactly $(f_1,...,f_m)\in M$.
\end{proof}

Now we know $M$ is a free $\C[x,y]$-module. To see its dimension,
let $\ds{f=\prod_{e_{ij}\in E}\alpha(e_{ij})}$. The dimension of $M$ as module over $\C[x,y]$
equals to the dimension of $M_{f}$ as a module over $\C[x,y]_{f}$.
But $M_f$ is generated by $(1,0,...,0),(0,1,...,0),...,(0,0,...,1)$
as a $\C[x,y]_f$-module. So $\mbox{dim} M=m$. This complete the proof of the theorem.
\end{proof}

The similar result does not hold in higher dimensions in general, even if we assume $\Ga$ is a regular graph and $\alpha$ is induced from a map $\phi: V\ra \R^l$. An easy counterexample is the following.
\begin{example}
Consider $\Ga=(V, E)$ given by $V=\{v_1, v_2, v_3, v_4\}$ and $E=\{e_{12}, e_{23}, e_{34}, e_{14}\}$. Note that $\Ga$ is a regular graph of degree $2$.  Define $\phi : V\ra \R^3$ by 
\[\phi(v_1)=(0,0,0), \phi(v_2)=(1,0,0), \phi(v_3)=(1,1,0), \phi(v_4)=(1,1,1).\]
Then $\phi$ induces $\alpha: E\ra \C[x,y,z]_1$ given by
\[\alpha(e_{12})=x, \alpha(e_{23})=y, \alpha(e_{34})=z, \alpha(e_{14})=x+y+z.\]
Then $H_{\T^3}^{*}(\Gamma,\alpha)$ as a module over $\C[x,y,z]$ is generated by
\[(1,1,1,1), (0,x,x+y,x+y+z), (0,xy,0,0),(0,0,yz,0), \mbox{\ and\ } (0,xz,xz,0).\]
This is not free as a $\C[x,y,z]$-module, since the generators are related by 
\[y(0,xz,xz,0)=z(0,xy,0,0)+x(0,0,yz,0).\]
\end{example}


This naturally leads to the following question.

\begin{question}
Given a regular graph $\Ga=(V,E)$ together with an {\it axial function} $\al: E\ra \C[x_1, ..., x_{l}]_1$ in the sense of
definition 1 in 2.1 in \cite{GZ:graph}, is the graph cohomology
necessarily a free module over $\C[x_1, ..., x_{l}]$? If not, what are the combinatorial criteria
on $(\Gamma,\alpha)$ that guarantee $H_{\T^l}^{*}(\Gamma,\alpha)$ to be free?
\end{question}

\section{\bf Properties of characteristic number $c_i$} \label{sec: characteristic}
In this section we study some basic properties of the characteristic numbers, whose definition was given in Definition~\ref{def:characteristic}. These numbers do not satisfy Poincare  duality in general, but we will prove a weaker version of it. We will discuss how does the ``top characteristic number" behaves for regular graphs. Then we discuss their relations with the combinatorial Betti numbers.  After that, we will compute the characteristic numbers for complete graphs, which are (graph theoretically) GKM graphs for complex projective spaces. Finally, we prove an analogue of K\"unneth formula. 

\begin{notation}
Continuing from Setup~\ref{setup1}, the vertices of $\Ga$ are labeled as $v_1, v_2, ..., v_m$. The edge connecting $v_{i}$ and $v_{j}$ will be denoted by $e_{ij}$. We do not distinguish between $e_{ij}$ and $e_{ji}$, but most of the time we will use the smaller number as the first index. We will use $\lambda(v_i)$ to denote the degree of $v_i$, i.e., the number of edges containing $v_{i}$. 

The characteristic number $c_{i}$ is not affected by a linear automorphism of $\R^2$ composed with $\phi: V\ra \R^2$, so without loss of generality, we may assume $\phi(v_{i})$ and $\phi(v_{j})$ have different second component for any $i \neq j$. Also we notice the definition of graph cohomology remains unchanged if we scale $\alpha$, so if we assume $\phi(v_{i})=(p_{i}, q_{i})$, $\phi(v_{j})=(p_{j}, q_{j})$ and let $a_{ij}=-\frac{p_{j}-p_{i}}{q_{j}-q_{i}}$, we may redefine $\alpha$ as
\[\alpha(e_{ij})=y-a_{ij}x.\]

For any positive integers $p\leq q$, we will use $\b_{p}^{q}$ to denote the $p$-th standard basis vector in $\C^{q}$, i.e., the $p$-th entry of $\b_{p}^{q}$ is $1$  and it is the only non-zero entry.

For $1\leq i<j \leq m$, we let 
\[\begin{array}{cl}
\v_{ij}^{0} & =\b_{i}^m-\b_{j}^m=(0,\cdots, 1, 0,\cdots, -1, 0,\cdots, 0)\in \C^m;\\[0.5ex]
\v_{ij}^1&=\b_{i}^{2m}-\b_{j}^{2m}+a_{ij}\b_{i+m}^{2m}-a_{ij}\b_{j+m}^{2m}\\[0.5ex]
&=(0,\cdots, 1,0\cdots, -1,0,\dots,0,0,\cdots, a_{ij},0,\cdots, -a_{ij},0, \cdots, 0)\\[0.5ex]
&=(\v_{ij}^0,\ a_{ij}\v_{ij}^0)\in \C^{2m};\\[0.5ex]
\v_{ij}^2&=\b_{i}^{3m}-\b_{j}^{3m}+a_{ij}\b_{i+m}^{3m}-a_{ij}\b_{j+m}^{3m}+a_{ij}^2\b_{i+2m}^{3m}-a_{ij}^2\b_{j+2m}^{3m}\\[0.5ex]
&=(\v_{ij}^1,\ a_{ij}^2\v_{ij}^{0})\in \C^{3m};
\end{array}\]
and so forth, where we use $(\bf{u}, \bf{v})$ to denote concatenation of two vectors $\bf{u}$ and $\bf{v}$.
In general, 
\[
\begin{array}{cl}
\v_{ij}^k&=\b_{i}^{(k+1)m}-\b_{j}^{(k+1)m}+a_{ij}\b_{i+m}^{(k+1)m}-a_{ij}\b_{j+m}^{(k+1)m}+\cdots+a_{ij}^k\b_{i+km}^{(k+1)m}-a_{ij}^k\b_{j+km}^{(k+1)m}\\[0.5ex]
&=(\v_{ij}^{k-1},\ a_{ij}^{k}\v_{ij}^{0})\in \C^{(k+1)m}.
\end{array}\]

Denote by $M_{k}(\Ga)$ the matrix of size $|E|\times (k+1)m$, whose rows are indexed by $E$ and the row corresponding to $e_{ij}$ is $\v_{ij}^k$. 

Let $r_{k}(\Ga)=\mbox{rank} M_{k}(\Ga)$. This is the dimension of the vector space spanned by the vectors $\{\v_{ij}^k\big| e_{ij}\in E\}$. We let $r_{-1}(\Ga)=0$.

Let $s_{k}(\Ga)=|E|-r_{k}(\Ga)$. This is the dimension of the vector space of linear relations among the vectors 
$\{\v_{ij}^k\big| e_{ij}\in E\}$. And we let $s_{-1}(\Ga)=|E|$.

Throughout the rest of the paper, we will use the following notations as we have introduced in Setup~\ref{setup1} and just above:
\[\Ga,\ V,\  E,\  m;\]
\[v_{1}, v_2, ..., v_{m},\ \lambda(v_i),\  e_{ij},\mbox{\ for\ } 1\leq i, j\leq m; \]
\[\phi,\  \alpha,\  a_{ij},\  \v_{ij}^k,\  \mbox{\ for\ } 1\leq i, j\leq m \mbox{\ and\ } k\geq 0;\]
\[M_{k}(\cdot),\ c_{k}(\cdot)\mbox{\ for\ } k\geq 0;\ \ r_{k}(\cdot), s_{k}(\cdot),\mbox{\ for\ } k\geq -1,\]
 where $\cdot$ in the last row can denote any graph in the sense of Setup~\ref{setup1}. 
\end{notation}

\begin{remark}
The matrix $M_{1}(\Ga)$ is closely related to the notion of {\it rigidity matrix} in rigidity theory.  More information about rigidity matrix and rigidity theory can be found in \cite{GSS}. The $``k=1"$ version of some statements in this paper have been proved in \cite{LUO:rigidity}, in slightly different language.
\end{remark}

\begin{figure}[!ht]
\includegraphics[height=60mm]{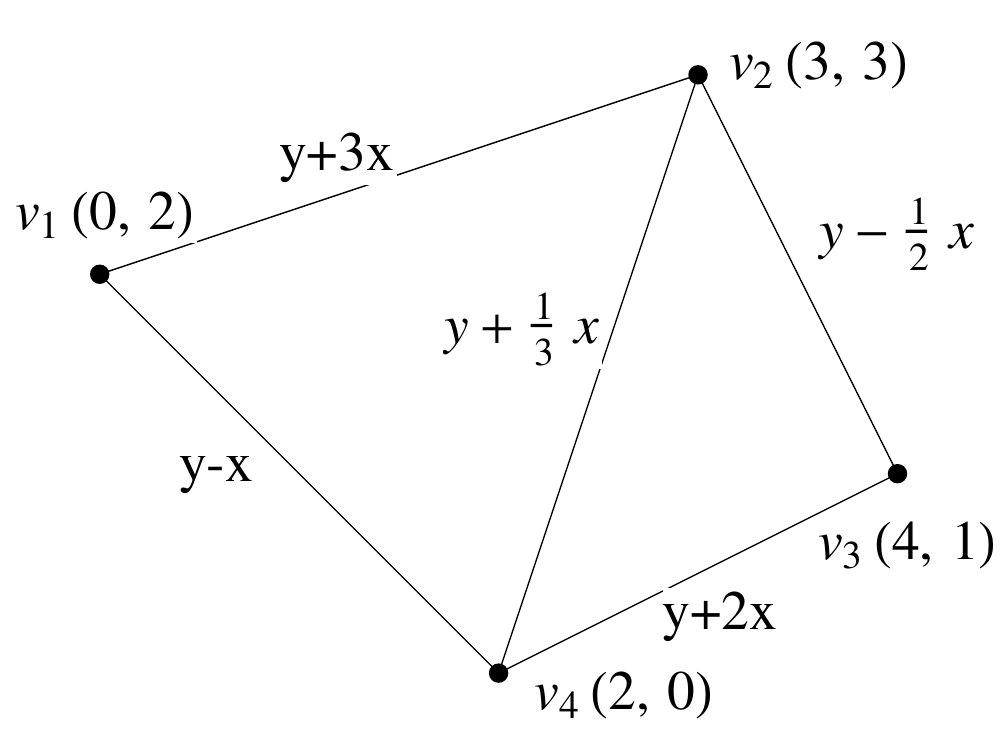}
\caption{An example of $\Ga, \phi$ and $\alpha$}
\label{fig:firstexample}
\end{figure}

\begin{example}\label{ex:first}
Figure~\ref{fig:firstexample} shows an easy example of a graph. On the graph we marked the coordinates of image of $\phi$ and also the image of $\alpha$.  
The edge set is $E=\{e_{12}, e_{14}, e_{23}, e_{24}, e_{34}\}$. So $M_{0}(\Ga)$ is a $5\times 4$ matrix with rows indexed by $E$ and column indexed by $V$. We can write it out explicitly
\[
M_0(\Ga)=\ \ \ \bordermatrix{\text{}&v_1&v_2&v_3 &v_4\cr
                e_{12}&1 &  -1  & 0 & 0\cr
                e_{14}& 1  &  0 & 0 & -1\cr
                e_{23}& 0 & 1 & -1 & 0\cr
                e_{24}& 0  &   1 &0 & -1\cr
                e_{34}& 0 &  0& 1& -1}.
\]
We can compute $M_{1}(\Ga), M_{2}(\Ga),$ and so forth by adding columns. As labeled in Figure~\ref{fig:firstexample}, we can see $a_{12}=-3, a_{14}=1, a_{23}=\dfrac{1}{2}, a_{24}=-\dfrac{1}{3}$ and $a_{34}=-2$, so $M_{2}(\Ga)$ is
\[
M_2(\Ga)=\bordermatrix{\text{}&v_1^0&v_2^0&v_3^0 &v_4^0\ \ &v_1^1&v_2^1&v_3^1 &v_4^1\ \ &v_1^2&v_2^2&v_3^2 &v_4^2\cr
                e_{12}&1 &  -1  & 0 & 0\ \ &a_{12}&-a_{12}&0&0\ \ &a_{12}^2&-a_{12}^2&0&0\cr
                e_{14}& 1  &  0 & 0 & -1\ \ &a_{14}&0&0&-a_{14}\ \ &a_{14}^2&0&0&-a_{14}^2\cr
                e_{23}& 0 & 1 & -1 & 0\ \ &0&a_{23}&-a_{23}&0\ \ &0&a_{23}^2&-a_{23}^2&0\cr
                e_{24}& 0  &   1 &0 & -1\ \ &0&a_{24}&0&-a_{24}\ \ &0&a_{24}^2&0&-a_{24}^2\cr
                e_{34}& 0 &  0& 1& -1\ \ &0&0&a_{34}&-a_{34}\ \ &0&0&a_{34}^2&-a_{34}^2}.
\]
One can show that $r_{0}(\Ga)=3$ and $r_{k}(\Ga)=5$ for all $k\geq 1$. It follows that $s_{0}(\Ga)=2$ and $s_{k}(\Ga)=0$ for all $k\geq 1$.
\end{example}

\begin{proposition}\label{prop:c_i and r_i}
The characteristic numbers can be computed as follows:

\begin{align}
c_{0}(\Ga)&=\pi_{0}(\Ga), \ \ \mbox{the\ number\ of\ connected\ components\ of\ }\Ga, \label{eq:pi_0}\\[0.2ex]
&=m-r_{0}(\Ga)=m-|E|+s_{0}(\Ga);\label{eq:c_0}\\[1.5ex]
c_{k}(\Ga)&=2r_{k-1}(\Ga)-r_{k}(\Ga)-r_{k-2}(\Ga) \label{eq:c_k} \\[0.2ex]
&=s_{k}(\Ga)+s_{k-2}(\Ga)-2s_{k-1}(\Ga), \forall k\geq 1.
\end{align}

\end{proposition}
\begin{proof}
Let $f=(f_{1},f_{2}, ..., f_{m})\in H_{\T}^{0}(\Ga)$, then it follows from the definition of graph cohomology that this is the case if and only if $f_{i}=f_{j}\in \C$ whenever $v_{i}$ and $v_{j}$ are in the same connected component of $\Ga$. So $c_{0}$ equals to the number of connected components of $\Ga$, which we denote by $\pi_{0}(\Ga)$.  This proves \eqref{eq:pi_0}.

For any $k\geq 0$, assume 
\[f=(\sum_{n=0}^{k}z_{1n}x^{k-n}y^n, \sum_{n=0}^{k}z_{2n}x^{k-n}y^n, ..., \sum_{n=0}^{k}z_{mn}x^{k-n}y^n)\in H_{\T}^{k}(\Ga),\] where $z_{st}\in \C$ for $1\leq s\leq m, 0\leq t\leq k$. Then by the definition of graph cohomology, this is equivalent to 
\[(y-a_{ij}x)\bigg| (\sum_{n=0}^{k}z_{in}x^{k-n}y^n)-(\sum_{n=0}^{k}z_{jn}x^{k-n}y^n), \forall e_{ij}\in E,\]
which in turn means when we substitute $y$ with $a_{ij}x$ in $\ds{ (\sum_{n=0}^{k}z_{in}x^{k-n}y^n)-(\sum_{n=0}^{k}z_{jn}x^{k-n}y^n) }$, we should get zero polynomial:
\[\begin{array}{cl}
0&=\ds{(\sum_{n=0}^{k}a_{ij}^nz_{in}x^k)-(\sum_{n=0}^{k}a_{ij}^nz_{jn}x^k)}\\
&=\ds{\left(\sum_{n=0}^ka_{ij}^n(z_{in}-z_{jn})\right)x^k}
\end{array}
\]
So $\ds{\sum_{n=0}^ka_{ij}^n(z_{in}-z_{jn})=0}$ and it can rewritten as 
\[\v_{ij}^{k}\cdot (z_{10},z_{20}, ..., z_{m0}, z_{11}, z_{21}, ..., z_{m1}, z_{12}, ..., z_{1k}, z_{2k}, ..., z_{mk})=0.\]
So 
\[\mbox{dim}H_{\T}^{k}(\Ga)=(k+1)m-r_{k}(\Ga).\]
By definition of $c_{i}$'s, we have
\[\mbox{dim}H_{\T}^{k}(\Ga)=\sum_{n=0}^{k}(k+1-n)c_{n},\]
so 
\begin{equation}\label{eq:recursive}
\sum_{n=0}^{k}(k+1-n)c_{n}=(k+1)m-r_{k}(\Ga).
\end{equation}

In particular, when $k=0$, this gives us
\[c_{0}=m-r_{0}.\]
This proves \eqref{eq:c_0}.

Now the formula~\eqref{eq:c_k} for $c_{k}, k\geq 1$, can be proved inductively. 

The base case $k=1:$  

By equation~\eqref{eq:recursive}, 
\[c_{1}=2m-r_{1}-2c_0=2m-r_{1}-2(m-r_0)=2r_{0}-r_{1}=2r_{0}-r_{1}-r_{-1}=s_{1}+s_{-1}-2s_{0}.\]
Now for $k>1$, we assume the formula~\eqref{eq:c_k} holds for smaller $k$. Then by equation~\eqref{eq:recursive}, we have
\[\begin{array}{cl}
c_{k}&=\ds{(k+1)m-r_{k}-\sum_{n=0}^{k-1}(k+1-n)c_{n}}\\
&=\ds{(k+1)m-r_{k}-\sum_{n=1}^{k-1}(k+1-n)(2r_{n-1}-r_{n-2}-r_{n})-(k+1)(m-r_0)}.
\end{array}\]
Upon straightforward simplification, this becomes exactly what we want:
\[c_{k}=2r_{k-1}-r_{k}-r_{k-2}=s_{k}+s_{k-2}-2s_{k-1}.\]
\end{proof}
\begin{example}
We continue with Example~\ref{ex:first}. Applying Proposition~\ref{prop:c_i and r_i} to this example gives $c_{0}(\Ga)=1, c_{1}(\Ga)=1$ and $ c_{2} (\Ga)=2$.
\end{example}

The following lemma will be used repeatedly throughout the paper. We call it the {\it Deleting Lemma}.
\begin{lemma}[Deleting Lemma]\label{lemma:deleting}
Assume $v_{t}\in V$ is of degree less than or equal to $k+1$, i.e., $\lambda(v_t)\leq k+1$. Denote by $E_{v_{t}}$ the set of edges containing $v_t$. Define $\Ga'=(V', E')$ by $V'=V\backslash\{v_t\}$ and $E'=E\backslash\{E_{v_{t}}\}$. In other words, $\Ga'$ is the graph obtained by deleting $v_t$ and the edges containing it. Then $s_{k}(\Ga)=s_{k}(\Ga')$.
\end{lemma}
\begin{proof}
For simplicity, we write $d$ for $\lambda(v_{t})$. Without loss of generality, we may assume $t=1$ and the $d$ edges containing $v_{t}$ are $e_{12}, e_{13}, ..., $and $e_{1,d+1}$. Now assume there is a linear relation among $\{\v_{ij}^{k}\big| e_{ij}\in E\}$:
\begin{equation}\label{eq:relation1}
\sum_{e_{ij}\in E}u_{ij}\v_{ij}^k=\0.
\end{equation}
If we restrict our attention to the components $1, m+1, ..., km+1$, then because the following matrix has full row rank (the assumption of $\phi$ in general position guarantees that $a_{1i}\neq a_{1j}$ for $i\neq j$),
\[\left(
\begin{array}{ccccc}
1& a_{12}& a_{12}^2&\ldots&a_{12}^k\\
1& a_{13}& a_{13}^2&\ldots&a_{13}^k\\
\vdots&\vdots&\vdots&\ddots&\vdots\\
1&a_{1,d+1}&a_{1,d+1}^2&\ldots&a_{1,d+1}^{k}
\end{array}
\right)
\]
we conclude that $u_{1j}=0$ for all $2\leq j\leq d+1$. So 
\[\sum_{e_{ij}\in E'}u_{ij}\v_{ij}^{k}=\0.\]
This means \eqref{eq:relation1} is in fact a linear relation among $\{\v_{ij}^{k}\big| e_{ij}\in E'\}$. So $s_{k}(\Ga)=s_{k}(\Ga')$.
\end{proof}

\begin{corollary}\label{cor:delete}
Assume $v_{t}\in V$. Denote by $E_{v_{t}}$ the set of edges containing $v_t$. Define $\Ga'=(V', E')$ by $V'=V\backslash\{v_t\}$ and $E'=E\backslash\{E_{v_{t}}\}$, then \[0\leq s_{k}(\Ga)-s_{k}(\Ga')\leq \max(\lambda(v_{t})-k-1,0).\]
\end{corollary}
\begin{proof}
It immediately follows from the Deleting Lemma~\ref{lemma:deleting}.
\end{proof}

\begin{lemma}\label{lemma:vanishing}
Given $\Ga=(V,E)$, when $\ds{k\geq \max_{i}\lambda(v_i)-1}$, we have $r_{k}(\Ga)=|E|$, $s_{k}(\Ga)=0$. And then it follows from Proposition~\ref{prop:c_i and r_i} that $c_{k+2}(\Ga)=0$.
\end{lemma}
\begin{proof}
The repeated application of the Deleting Lemma will give $s_{k}(\Ga)=0$. Hence $r_{k}(\Ga)=|E|-s_{k}(\Ga)=|E|$ and $c_{k+2}=s_{k+2}-2s_{k+1}+s_{k}=0$.
\end{proof}

\begin{proposition}\label{prop:sum c_i and ic_i}
Given $\Ga=(V,E)$, we have 
\begin{equation}\label{eq:sum c_i}
\sum_{i=0}^{\infty}c_{i}(\Ga)=m\ ,
\end{equation}
\begin{equation}\label{eq:sum ic_i}
\sum_{i=0}^{\infty}ic_{i}(\Ga)=|E|.
\end{equation}
\end{proposition}
\begin{proof}
Equation~\eqref{eq:sum c_i} is part of the content of Theorem~\ref{theorem:freeness}.  As for $\eqref{eq:sum ic_i}$, according to Lemma~\ref{lemma:vanishing}, we can pick $k\in \N$, such that $r_{k}(\Ga)=|E|$ and $c_{i}=0$ for all $i\geq k+1$. Then by formula~\eqref{eq:recursive}, 
\[\sum_{i=0}^{k}(k+1-i)c_{i}=(k+1)m-r_{k}=(k+1)m-|E|.\]
So 
\[\sum_{i=0}^{\infty}ic_i=\sum_{i=0}^{k}ic_i=\sum_{i=0}^{k}(k+1)c_{i}-(k+1)m+|E|=|E|.\] 
\end{proof}

\begin{remark}
When $\Ga$ is a regular graph of degree $d$, then it follows from Proposition~\ref{prop:sum c_i and ic_i} that the average degree of a set of generators of $H_{\T}^{*}(\Ga)$ is $\dfrac{d}{2}$. In this sense, it is a weaker version of Poincare duality.
\end{remark}

\begin{proposition}\label{prop:regular graph}
Assume $\Ga=(V,E)$ is a connected regular graph of degree $d$, then $c_{i}=0$ for $i>d$ and $c_{d}=0$ or $1$. 
\end{proposition}
\begin{proof}
The first part of the proposition is just a special case of Lemma~\ref{lemma:vanishing}. 

As for the second part, pick any edge $e_{ij}$, define $\Ga'=(V, E')$ by $E'=E\backslash \{e_{ij}\}$.  Then we have $s_{d-2}(\Ga)\leq s_{d-2}(\Ga')+1$.  By repeatedly applying the Deleting Lemma to $\Ga'$, we get $s_{d-2}(\Ga')=0$. So $s_{d-2}(\Ga)\leq 1$.   

By Lemma~\ref{lemma:vanishing}, we have $s_{d-1}(\Ga)=s_{d}(\Ga)=0$.  So \[c_{d}(\Ga)=s_{d}(\Ga)+s_{d-2}(\Ga)-2s_{d-1}(\Ga)\leq 1.\] 
\end{proof}
\begin{remark}
As the GKM graph of GKM manifolds are regular graphs, Proposition~\ref{prop:regular graph} can be viewed as a weaker version of the fact that the top geometric Betti number of a connected compact oriented manifold is $1$. If the connected regular graph $\Ga$ of degree $d$ is also equipped with an axial function in the sense of \cite{GZ:graph}, then one can show by Theorem 2.2 in \cite{GZ:graph} that $c_{d}(\Ga)=1$ and $\ds{(\prod_{e_{1j}\in E}\alpha(e_{1j}),0, 0, \cdots, 0)}$ can be taken as the generator of $H_{\T}^{*}(\Ga)$ in degree $d$. 
\end{remark}
\begin{figure}[!ht]
\includegraphics[height=60mm]{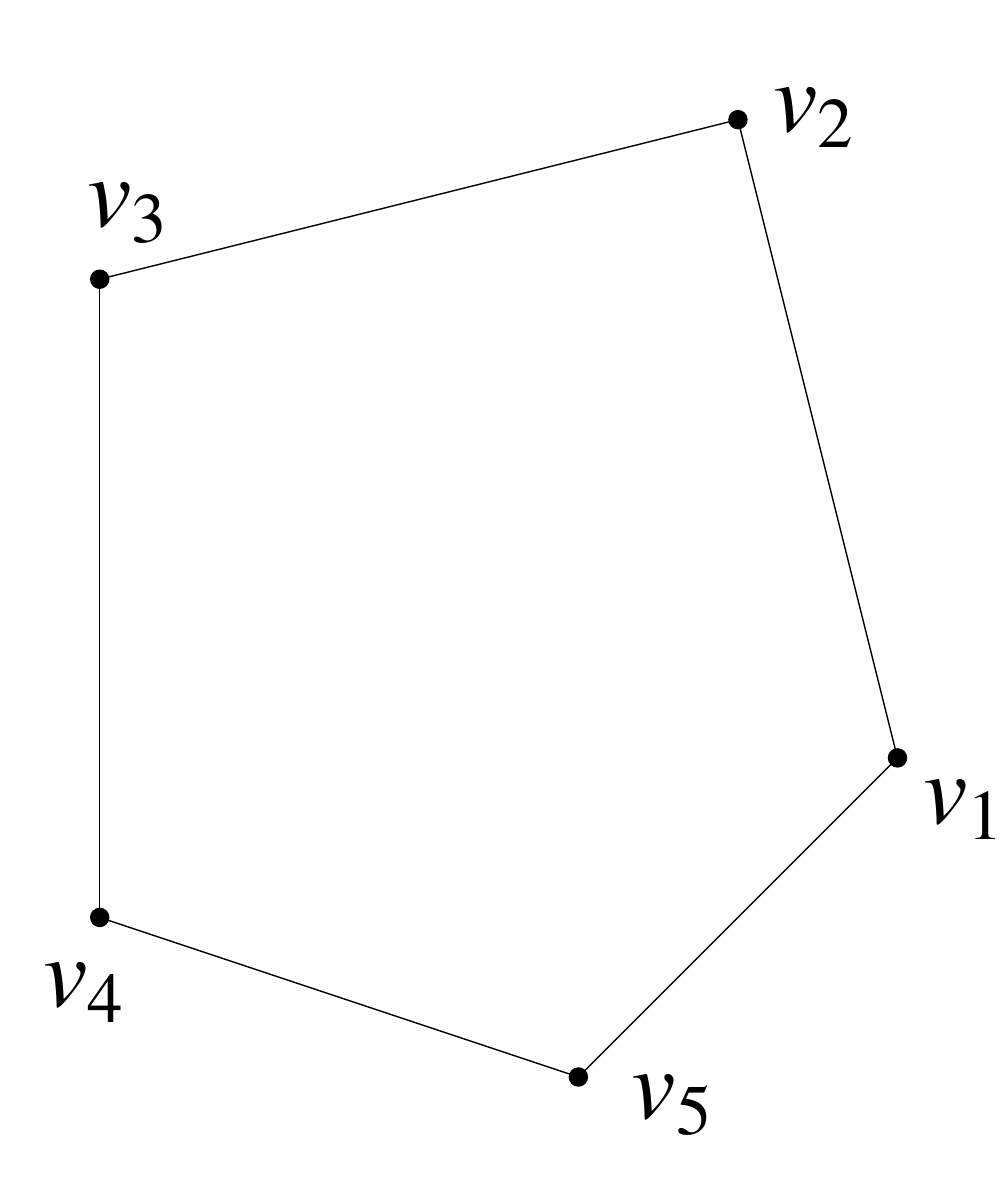}
\caption{A regular graph of degree 2}
\label{fig:2-regular}
\end{figure}
\begin{example}\label{ex:2-regular}
Figure~\ref{fig:2-regular} shows a regular graph of degree $2$ whose vertices are in general position. By Proposition~\ref{prop:regular graph}, $c_i=0$ for $i\geq 3$. Also we know $c_0=1$ as the graph is connected. Then the weaker version of Poincare duality, Proposition~\ref{prop:sum c_i and ic_i}, forces $c_2$ to be $1$, hence $c_1=3$. 

In general, for regular graph of degree $2$ and $m$ vertices in general position, we have $c_0=c_2=1$ and $c_1=m-2$.
\end{example}

Next we study the relations between the characteristic numbers and the combinatorial Betti numbers. Proposition~\ref{prop:characteristic vs Betti} provides a generalization of inequality (2.28) in \cite{GZ:graph} in dimension two.

\begin{proposition}\label{prop:characteristic vs Betti}
Given $\Gamma=(V,E)$, order
the vertices as $v_1,\cdots, v_{m}$. Define the index of $v_i$ as
\[\mu_{i}=\bigg|\{j\in \N\big| j>i \mbox{\ and\ } e_{ij}\in E\}\bigg|,\]
and let 
\[b_k=\bigg|\{i\in \N\big| \mu_{i}=k\}\bigg|.\]
Then
 \[\sum_{i=0}^{k}(k+1-i)c_i\leq \sum_{i=0}^{k}(k+1-i)b_i,\mbox{\ for\ any} \ k\geq 0.\]
\end{proposition}
\begin{proof}
 First observe that $\displaystyle \sum_{i=0}^{\infty}b_i=|V|=m$ because each vertex is counted exactly once.
 Secondly, we claim that $\displaystyle \sum_{i=0}^{\infty}ib_i=|E|$. To see this, count the edge set in the following way. First look at a vertex $v_{i}$ with $\mu_{i}=0$, and see how many edges connect $v_{i}$ with some vertex $v_{j}$ whose subscript $j$ is larger than $i$.  The answer must be $0$ by the definition of $\mu_{i}$. Then we look at a vertex $v_i$ with $\mu_{i}=1$, in this case, there should be $1$ edge connecting $v_{i}$ with some vertex of larger subscript. Repeat this process for each vertex. Then altogether we will have counted $\displaystyle \sum_{i=0}^{\infty}ib_i$ many edges. On the other hand, we see that every edge in $E$ is counted once and only once. So $\displaystyle \sum_{i=0}^{\infty}ib_i=|E|$.
 
  Denote by $\Ga'$ the subgraph of $\Ga$ obtained by deleting $v_{1}$ and the edges containing $v_{1}$. Applying Corollary~\ref{cor:delete} to $v_1$, we get
\[s_{k}(\Ga)\leq s_{k}(\Ga')+\max(\mu_{1}-(k+1), 0).\]
Applying Corollary~\ref{cor:delete} to $v_{2}, v_{3}, ..., v_{m}$ in order, we get
\[\begin{array}{cl}
s_{k}(\Ga)&\ds{\leq \sum_{j=1}^{m}\max(\mu_{j}-(k+1), 0)}\\[0.4ex]
&\ds{=\sum_{i=k+2}^{\infty}(i-(k+1))b_{i}}\\[0.4ex]
&\ds{=\sum_{i=0}^{\infty}(i-(k+1))b_i-\sum_{i=0}^{k}(i-(k+1))b_i}\\[0.4ex]
&\ds{=\sum_{i=0}^{\infty}ib_i-(k+1)\sum_{i=0}^{\infty}b_i+\sum_{i=0}^{k}(k+1-i)b_i}\\[0.4ex]
&=|E|-(k+1)m+\ds{\sum_{i=0}^{k}(k+1-i)b_i}.
\end{array}\]
So 
 $$\begin{array}{cl}
 \ds{\sum_{i=0}^{k}(k+1-i)c_i}&=(k+1)m-r_k\\[0.2ex]
         &=(k+1)m-(|E|-s_k)\\[0.2ex]
         &\leq (k+1)m-|E|+|E|-(k+1)m+\ds{\sum_{i=0}^{k}(k+1-i)b_i}\\[0.2ex]
         &=\ds{\sum_{i=0}^{k}(k+1-i)b_i},
\end{array}$$
where  the first equality is equation \eqref{eq:recursive}.
\end{proof}

We next compute the characteristic numbers of a complete graph. They turn out to be what we would expect, precisely the Betti numbers of complex project space.

\begin{definition}
Define $\omega\in H_{\T}^{1}(\Ga)$ by $\omega=(x_{1}x+y_{1}y, x_{2}x+y_{2}y, ..., x_{m}x+y_{m}y)$, where $(x_{i}, y_{i})=\phi(v_{i})$. We call this the {\it equivariant symplectic form} of $\Ga$. 

If we denote by $I$ the ideal of $H_{\T}^{*}(\Ga)$ generated by $(x, x, ..., x)$ and $(y, y, ..., y)$, then the quotient ring $H_{\T}^{*}(\Ga)/I$, which we will denote by $H^{*}(\Ga)$, is a vector space of dimension $m$.  A module basis of $H_{\T}^{*}(\Ga)$ descends to a vector space basis of $H^{*}(\Ga)$. The image of $\omega$ in $H^{*}(\Ga)$ is denoted by $\tilde{\omega}$ and is called the {\it ordinary symplectic form} of $\Ga$. As graded $\C[x,y]$-modules, we have $H_{\T}^{*}(\Ga)\cong H^{*}(\Ga)\otimes_{\C}\C[x,y]$.

 If $\Ga'=(V', E')$ is a subgraph of $\Ga=(V,E)$, then $\Ga'$ comes with $\phi'$ and $\alpha'$ that is induced by $\phi$ and $\alpha$. If we denote by $i$ the natural inclusion map $i:\Ga'\ra \Ga$, where the map can be viewed as both maps between edges sets and between vertex sets, then it induces a graded $\C[x,y]$-algebra homomorphism $i^*:H_{\T}^{*}(\Ga)\ra H_{\T}^{*}(\Ga')$.  This map descends to a graded ring homomorphism $\tilde{i}^{*}: H^{*}(\Ga)\ra H^{*}(\Ga')$.
\end{definition}

\begin{proposition}\label{prop:complete}
Given $\Ga=(V,E)$ a complete graph on $m$ vertices, we have $c_{0}=c_{1}=c_{m-1}=1$ and $c_{k}=0$ for $k\geq m$. Moreover, $\{\omega^i\big| 0\leq i\leq m-1\}$ forms a basis of the free module $H_{\T}^{*}(\Ga)$ over $\C[x,y]$. 
\end{proposition}
\begin{proof}
Assume $\phi(v_{i})=(x_i, y_i), 1\leq i \leq m$. First we observe that a translation of $\phi$ to $\phi'=\phi+(p,q)$, where $(p, q)\in \R^2$, will not affect the graph cohomology at all. And the classes $\{\omega^i\big| 0\leq i\leq m-1\}$ forming a module basis is equivalent to the classes $\{(\omega+\gamma)^i\big| 0\leq i\leq m-1\}$ forming a module basis, where $\gamma=(px+qy, px+qy, ..., px+qy)$.
So without loss of generality, we may assume $x_i\neq 0$ for all $i$ and $\dfrac{y_i}{x_i}\neq \dfrac{y_j}{x_j}$ for all $i\neq j$.

We prove the statement by induction on $m$. 

When $m=1$, the statement clearly holds. Now assume we have proved that for $m=k$ and $\Ga$ a complete graph on $m$ vertices, the classes $\{\omega^i\big| 0\leq i\leq m-1\}$ forms a module basis of $H_{\T}^{*}(\Ga)$. And we consider the case $m=k+1$.

Denote by $\Ga'$ the complete graph on $v_{1}, v_{2}, ..., v_{k}$. This is a subgraph of $\Ga$ and we denote by $i$ the inclusion map.  Now $\omega$ is the equivariant symplectic form of $\Ga$, and we notice that $i^*(\omega)$ is the equipvariant symplectic form of $\Ga'$. For any $0\leq i\leq k-1$, by our induction hypothesis, we have $\tilde{i}^*(\tilde{\omega}^i)\neq 0 \in H^{i}(\Ga')$, so $\tilde{\omega}^i\neq 0\in H^{i}(\Ga)$, hence $c_{i}(\Ga)\geq 1$.

Then the relations 
\[\sum_{i=0}^{k}c_{i}(\Ga)=k+1\]
and 
\[\sum_{i=0}^{k}ic_{i}(\Ga)=|E|=\frac{k(k+1)}{2}\]
force $c_{0}(\Ga)=c_{1}(\Ga)=\cdots=c_{k}(\Ga)=1$ and $c_{i}(\Ga)=0$ for $i\geq k+1$.

Now to complete the induction step, we only need to show $\tilde{\omega}^k\neq 0 \in H^{k}(\Ga)$. Assume this is not the case, then there exists homogeneous polynomials $f_{i}(x,y)$ of degree $i$, $1\leq i\leq k$, such that 
\[\omega^k = f_{k}+f_{k-1}\omega+\cdots+f_{1}\omega^{k-1}\in H_{\T}^{k}(\Ga),\]
where we have identified the polynomial $f_{i}$ with $(f_{i}, f_{i}, ..., f_{i})\in H_{\T}^{i}(\Ga)$.

This implies the following equation
\begin{equation}\label{eq:omega^k}
(ax+by)^k\equiv f_{k}(x,y)+f_{k-1}(x,y)(ax+by)+\cdots+f_{1}(x,y)(ax+by)^{k-1},\end{equation}
where $\equiv$ means the two sides are equal as polynomials in $x$ and $y$, holds for $(a,b)=(x_{i}, y_{i})=\phi(v_{i}), 1\leq i\leq k+1$.  We would show this is impossible. 

From ~\eqref{eq:omega^k}, we see that $(ax+by)\big| f_{k}(x,y)$ for all $(a,b)=(x_i, y_i), 1\leq i\leq k+1$. Because $\dfrac{y_i}{x_i}\neq \dfrac{y_j}{x_j}$ for all $i\neq j$, we have $\ds{\prod_{i=1}^{k+1}(x_{i}x+y_{i}y)\big|f_{k}(x,y)}$, this is only possible when $f_{k}(x,y)\equiv 0$. Then it follows from ~\eqref{eq:omega^k} that 
\[(ax+by)^{k-1}\equiv f_{k-1}(x,y)+f_{k-2}(x,y)(ax+by)+\cdots+f_{1}(x,y)(ax+by)^{k-2}\]
holds for $(a, b)=(x_i, y_i), 1\leq i\leq k+1$. In other words,
\[\omega^{k-1}=f_{k-1}+f_{k-2}\omega+\cdots+f_{1}\omega^{k-2}\in H_{\T}^{k-1}(\Ga).\]
This contradicts with that $\tilde{\omega}^{k-1}\neq 0\in H^{k-1}(\Ga)$.
\end{proof}

We conclude this section by proving an analogue of K\"unneth formula.
\begin{definition}
Given $\Ga_1=(V_1, E_1)$ and $\Ga_2=(V_2, E_2)$,  assume $V_1=\{v_1, v_2, ..., v_m\}$ and $V_2=\{u_1, u_2, ..., u_n\}$.
The (Cartesian) product graph $\Ga_1\Box \Ga_2=(V_3, E_3)$ is defined by $V_3=V_1\times V_2$ and two vertices $(v_i, u_s), (v_j, u_t)$ in $V_3$ are connected by an edge in $E_3$ if and only if 
 \[v_i=v_j\in V_1,\ u_s \mbox{\ and\ } u_t \mbox{\ are\ connected\ by\ an\ edge\ in\ } E_2, \]
 or
 \[u_s=u_t\in V_2, \ v_i \mbox{\ and\ } v_j \mbox{\ are \ connected\ by\ an\ edge\ in\ } E_1.\]
\end{definition}

\begin{proposition}
Given $\Ga_1=(V_1, E_1)$ and  $\Ga_2=(V_2, E_2)$, assume $V_1=(v_1, ..., v_m)$, $V_2=(u_1, ..., v_n)$ and $\phi_1:V_1\ra \R^2, \phi_2: V_2\ra \R^2$ are two moment maps in general position.  Assume $\Ga_3=\Ga_1\Box \Ga_2=(V_3, E_3)$ and $\phi_3: V_3=V_1\times V_2\ra \R^2$ defined by $\phi_3(v_i, u_s)=a\phi_1(v_i)+b\phi_2(u_s)$ is also in general position, where $a, b\in \R$ are both non-zero constants. Then the map $p$ defined as follows is an isomorphism of $\C[x,y]$-algebras:
\begin{equation}\label{eq:Kunneth}
\begin{array}{ccl}
p: H_{\T}^{*}(\Ga_1)\otimes_{\C[x,y]}H_{\T}^{*}(\Ga_2)&\longrightarrow &H_{\T}^{*}(\Ga_3)\\[1.5ex]
(f_1, f_2, ..., f_m)\otimes (g_1, g_2, ..., g_n)&\mapsto &(f_1g_1, f_1g_2, ..., f_1g_n, f_2g_1,..., f_mg_n),
\end{array}
\end{equation} 
where the vertices in $V_3$ are ordered as $(v_1, u_1), (v_1, u_2), ..., (v_1, u_n), (v_2, u_1), ..., (v_m, u_n)$. 

As an immediate consequence, we have 
\[c_{k}(\Ga_3)=\sum_{i=0}^{k}c_{i}(\Ga_1)c_{k-i}(\Ga_2).\]
\end{proposition}
\begin{proof}
It is straightforward to verify that the map $p$ defined by \eqref{eq:Kunneth} is an algebra homomorphism.
We are going to prove it is an isomorphism by induction on $|E_1|+|E_2|$.  If $|E_1|=0$ or $|E_2|=0$, the conclusion obviously holds. Now we assume $|E_1|>0, |E_2|>0$ and $p$ defined by \eqref{eq:Kunneth} is an isomorphism for smaller $|E_1|+|E_2|$.  

Suppose $v_i$ and $v_j$ are connected by an edge in $E_1$, which we will denote by $e_{v_i}^{v_j}$. Suppose $u_s$ and $u_t$ are connected by an edge  in $E_2$, which we will denote by $e_{u_s}^{u_t}$.  Assume $\alpha_k: E_k\ra \C[x,y]_1$ is induced by $\phi_k$ for $k=1, 2, 3$. Denote $\alpha_1(e_{v_i}^{v_j})$ by $f$ and $\alpha_2(e_{u_s}^{u_t})$ by $g$. We first observe that $f$ and $g$ cannot be multiples of each other, otherwise the three vertices $(v_i, u_s), (v_j, u_s), (v_j, u_t)$ in $V_3$ would be on the same line, contradicting the assumption that $\phi_3$ is in general position.  We know both $H_{\T}^{*}(\Ga_1)\otimes_{\C[x,y]}H_{\T}^{*}(\Ga_2)$ and $H_{\T}^{*}(\Ga_3)$ are free modules over $\C[x,y]$ of dimension $mn$, we now fix bases for both.  Then the map $p$
 can be represented by a matrix $M(p)$.  And $p$ is an isomorphism if and only if $M(p)$ is invertible as a matrix with coefficients in $\C[x,y]$, which is equivalent to $\det(M(p))$ being invertible in $\C[x,y]$.  
 
Define $E_1'=E_1\backslash\{e_{v_i}^{v_j}\}$ and $\Ga_1'=(V_1, E_1')$. We may define \[p':  H_{\T}^{*}(\Ga_1')\otimes_{\C[x,y]}H_{\T}^{*}(\Ga_2)\rightarrow H_{\T}^{*}(\Ga_1'\Box\Ga_2)\] 
the similar way we have defined $p$.  
 It follows from the induction hypothesis that $p'$ is an isomorphism.
 For any $\C[x,y]$-module $N$, we will use $N_f$ to denote $\ds{N\otimes_{\C[x,y]}\C[x,y]_f}$, the localization at the polynomial $f$. It follows from the definition of graph cohomology that \[H_{\T}^{*}(\Ga_1)_f=H_{\T}^{*}(\Ga_1')_f,\]
 and 
 \[H_{\T}^{*}(\Ga_1\Box\Ga_2)_f=H_{\T}^{*}(\Ga_1'\Box\Ga_2)_f.\] 
  Then the two maps
  \[p_f: H_{\T}^{*}(\Ga_1)_f\otimes_{\C[x,y]_f}H_{\T}^{*}(\Ga_2)_f\rightarrow H_{\T}^{*}(\Ga_1\Box\Ga_2)_f \]
  and 
  \[p'_f: H_{\T}^{*}(\Ga_1')_f\otimes_{\C[x,y]_f}H_{\T}^{*}(\Ga_2)_f\rightarrow H_{\T}^{*}(\Ga_1'\Box\Ga_2)_f\]
  can be identified. Because $p'$ is an isomorphism, so are $p'_f$ and $p_f$. 
  So $\det(M(p))$ is invertible in $\C[x,y]_f$.  By the same argument, we know that $\det(M(p))$ is also invertible in $\C[x,y]_g$. Since $f$ and $g$ are not multiples of each other, this is only possible when $\det(M(p))\in \C^{*}$. So $p$ is an isomorphism.  This completes the induction step.
  \end{proof}
  
\section{\bf Connectivity properties of GKM graphs of compact Hamiltonian GKM manifolds}\label{sec: connectivity}
In this section we prove two theorems about the connectivity of the GKM graph of a compact Hamiltonian GKM manifold.
\begin{definition}
A graph $\Ga=(V, E)$ is $k$-edge-connected for some $k\in \N$, if for any subset $F=\{e_{i_1j_1}, e_{i_2j_2}, ..., e_{i_tj_t}\}\subseteq E$ with $t<k$, the graph $\Ga'=(V, E')$ defined by $E'=E\backslash F$ is connected.
\end{definition}

\begin{definition}
A graph $\Ga=(V,E)$ is $k$-vertex-connected for some $k\in \N$, if for any subset $U=\{{v_{i_1}, v_{i_2}}, ..., v_{i_{t}}\} \subseteq V$ with $t<k$, the graph $\Ga'=(V', E')$ defined by $V'=V\backslash U, E'=E\backslash E_{U}$ is connected, where $E_{U}=\{e_{ij}\in E\big| v_{i}\in U \mbox{\ or\ }v_{j}\in U\}$.
\end{definition}

Now we are ready to state the main theorems in this section.
\begin{theorem}\label{theorem:edge connected}
Given $\Ga=(V,E)$ a connected regular graph of degree $d$, if $c_{d}(\Ga)=1$, then $\Ga$ is $d$-edge-connected. As a consequence, if a $2d$-dimensional compact connected Hamiltonian GKM manifold has a moment map in general position, then its GKM graph is $d$-edge-connected.
\end{theorem}

Let $\lceil \frac{d}{2} \rceil$ denote the least integer greater than or equal to $\frac{d}{2}$.
\begin{theorem}\label{theorem:vertex connected}
Given $\Ga=(V,E)$ a connected regular graph of degree $d$, if $c_{d}(\Ga)=1$, then $\Ga$ is $(\lceil \frac{d}{2} \rceil+1)$-vertex-connected. As a consequence, if a $2d$-dimensional compact connected Hamiltonian GKM manifold has a moment map in general position, then its GKM graph is $(\lceil \frac{d}{2} \rceil+1)$-vertex-connected. 
\end{theorem}

\begin{notation}
Given $\Ga=(V, E)$, we say a vector $\u=(u_1, u_2, ..., u_{(k+1)m})\in \C^{(k+1)m}$ vanishes on $v_{i}$, or $\u\big|_{v_{i}}=0$, if $u_{i}=u_{m+i}=\cdots=u_{km+i}=0$. We say $\u$ vanishes on $U\subseteq V$, or $\u\big|_{U}=0$, if $\u$ vanishes on every point in $U$. We denote the set of vectors in $\C^{(k+1)m}$ that vanishes on $V\backslash U$ by $W_{U}^{k}$.  We would use $W_{v_i}^k$ as shorthand for $W_{\{v_{i}\}}^k$ when the set has only one point. It is easy to see that $\v_{ij}^{k}\in W_{\{v_{i}, v_{j}\}}^k$.

We denote by $P_{U}^k: \C^{(k+1)m}\rightarrow W_{U}^{k}$ the natural projection which sets all coordinates corresponding to $V\backslash U$ to $0$. 

For any $U\subseteq V$, we use $K(U)$ to denote the edge set of the complete graph on $U$ and use $K_{U}=(U, K(U))$ to denote the complete graph itself. Note that when $U$ has only one element, $K(U)=\emptyset$ and $K_{U}$ is just a single vertex. 
\end{notation}

\begin{lemma}\label{lemma:edge vector self contained}
Assume $\Ga=(V,E)$ and $U\subseteq V$ is a non-empty subset. Then 
\[\langle \v_{ij}^{k}\big| e_{ij}\in E\rangle \cap W_{U}^{k}\ \subseteq \langle \v_{ij}^{k}\big|e_{ij}\in K(U)\rangle,\]
where $\langle S\rangle$ means the subspace of $\C^{(k+1)m}$ spanned by $S$. 
\end{lemma}
\begin{proof}
Recall we always have $V=\{v_1, v_2, ..., v_m\}$, now for simplicity we assume $U=\{v_1, v_2, ..., v_{n}\}$. It is enough if we could show 
\begin{equation}\label{eq:goal}
\langle\v_{ij}^{k}\big| e_{ij}\in K(V)\rangle\cap W_{U}^{k}\ =\langle\v_{ij}^{k}\big|e_{ij}\in K(U)\rangle.
\end{equation}
It is clear that in the above equation, the RHS is contained in the LHS. It remains to show the other direction. We divide the proof into three cases. 
\begin{enumerate}
\item[Case 1:] $n\geq k+1$. 
\begin{equation}\begin{array}{cl}
&\mbox{dim}\langle\v_{ij}^k\big|e_{ij}\in K(V)\rangle\\[0.3ex]
=&r_{k}(K_V)\\[0.3ex]
=&(k+1)m-c_{k}(K_V)-2c_{k-1}(K_V)-\cdots-(k+1)c_{0}(K_V)\\[0.3ex]
=&(k+1)m-\dfrac{(k+1)(k+2)}{2}. \label{eq:dimKV}
\end{array}\end{equation}
The first equality is because of the definition of $r_{k}$, the second equality is because of formula ~\eqref{eq:recursive}, the third equality is because of Proposition~\ref{prop:complete}.

It is obvious that 
\begin{equation}\label{eq:dimWU}\mbox{dim}W_{U}^{k}=(k+1)n.\end{equation}

Now for any $t\in \N$ that $n+1\leq t\leq m$ and any $i\in N$ that $1\leq i\leq k+1 \leq n$, we have $\v_{it}^{k}\in W^k_{\{v_i, v_t\}}$ and $P_{\{v_{t}\}}^k(\v_{it}^{k})\in <\v_{it}^k>+W_{v_i}^k$. More explicitly, we have 
\[P_{\{v_{t}\}}^k(\v_{it}^{k})=(0, ..., -1,0, ...,-a_{it}, 0, \cdots, -a_{it}^k, 0, ..., 0),\]
where the only nonzero entries of the vector are the $t, (m+t), ..., (km+t)$-th entries and they are $-1, -a_{it}, \cdots, -a_{it}^k$ respectively.  
Then the knowledge of Vandermonde matrix tells us $\{P_{\{v_{t}\}}^k(\v_{it}^{k})\big| 1\leq i\leq k+1\}$ form a basis for $W^k_{v_t}$. So
\[W^k_{v_t}\subseteq\sum_{i=1}^{k+1}(\langle\v_{it}^k\rangle+W_{v_i}^k)\ \subseteq\langle\v^k_{ij}\big|e_{ij}\in K(V)\rangle+W_{U}^k,\]
for $n+1\leq t\leq m$. 
So $\langle\v^k_{ij}\big|e_{ij}\in K(V)\rangle+W_{U}^k=W^k_V$ and 
\begin{equation}\label{eq:dimsum}
\mbox{dim}(\langle\v^k_{ij}\big|e_{ij}\in K(V)\rangle+W_{U}^k)=m(k+1).\end{equation}
Now we put \eqref{eq:dimKV}, \eqref{eq:dimWU}, \eqref{eq:dimsum} together to have
\begin{equation*}
\begin{array}{cl}
&\mbox{dim}(\langle\v_{ij}^k\big|e_{ij}\in K(V)\rangle\cap W_{U}^k)\\[0.3ex]
=&\mbox{dim}\langle\v_{ij}^k\big|e_{ij}\in K(V)\rangle+\mbox{dim}W_{U}^k- \mbox{dim}(\langle\v^k_{ij}\big|e_{ij}\in K(V)\rangle+W_{U}^k)\\[0.3ex]
=&(k+1)m-\dfrac{(k+1)(k+2)}{2}+(k+1)n-m(k+1)\\[1.2ex]
=&(k+1)n-\dfrac{(k+1)(k+2)}{2}\\[0.3ex]
=&\mbox{dim}\langle\v_{ij}^k\big|e_{ij}\in K(U)\rangle.
\end{array}
\end{equation*}
The last equality can be obtained exactly the same way as we obtained \eqref{eq:dimKV}.  This is enough to conclude \eqref{eq:goal}.
\item[Case 2:]$k=1, n=1$. We observe that each vector \[\u=(u_1,u_2, ..., u_{2m})\in\langle\v_{ij}^1\big|e_{ij}\in K(V)\rangle\] must satisfy the linear equations
\[\sum_{i=1}^m u_i=0\mbox{\ and\ } \sum_{i=m+1}^{2m}u_{i}=0.\]
This forces the intersection of $\langle\v_{ij}^1\big|e_{ij}\in K(V)\rangle$ with $W_{v_1}^1$ to be $\0$. So \eqref{eq:goal} holds since $K(\{v_{1}\})=\emptyset$.
\item[Case 3:]$k\geq 2, n < k+1$. 

For fixed $n\geq 1$, we are going to prove \eqref{eq:goal} using induction on $k$. If $n\geq 2$, we use $k=n-1$ as the base case. If $n=1$, we use $k=1$ as the base case. Note that the base cases all have already been dealt with in Case 1 and Case 2.  So now we have $k\geq 2, k>n-1$ and assume \eqref{eq:goal} holds for smaller $k$ (Again, note that $n$ is fixed first).

Assume $\u \in (\langle\v_{ij}^{k}\big| e_{ij}\in K(V)\rangle\cap W_{U}^{k})$, so we may write it as 
\begin{equation}\label{eq:u}
\u=\sum_{e_{ij}\in K(V)}u_{ij}\v_{ij}^k\in W_{U}^k
\end{equation}
for some $u_{ij}\in \C$. 

For any vector $\w=(w_1, w_2, ..., w_{(k+1)m})\in \C^{(k+1)m}$, let 
\[\w^{ini}=(w_1, w_2, ..., w_{km})\in \C^{km},\]
\[\w^{end}=(w_{m+1}, w_{m+2}, ..., w_{(k+1)m})\in \C^{km},\  \mbox{and}\]
\[\w^{mid}=(w_{m+1}, w_{m+2}, ..., w_{km})\in \C^{(k-1)m}.\]
Then in particular we have 
\[(\v^{k}_{ij})^{ini}=\v_{ij}^{k-1},\ (\v^{k}_{ij})^{end}=a_{ij}\v_{ij}^{k-1}, \mbox{\ and\ } (\v_{ij}^{k})^{mid}=a_{ij}\v_{ij}^{k-2}.\]
Then it follows from ~\eqref{eq:u} that
\[
\begin{array}{ccl}
\u^{ini}&=&\ds{\sum_{e_{ij}\in K(V)}u_{ij}\v_{ij}^{k-1}\in W_{U}^{k-1}} \mbox{\ and}\\
\u^{end}&=&\ds{\sum_{e_{ij}\in K(V)}u_{ij}a_{ij}\v_{ij}^{k-1}}\in W_{U}^{k-1}.
\end{array}
\]
Then by the induction hypothesis, we have 
\begin{equation}\label{eq:2}\begin{array}{ccl}
\u^{ini}&=&\ds{\sum_{e_{ij}\in K(U)}x_{ij}\v_{ij}^{k-1}} \mbox{\ and}\\
\u^{end}&=&\ds{\sum_{e_{ij}\in K(U)}y_{ij}\v_{ij}^{k-1}},
\end{array}
\end{equation}
for some $x_{ij}, y_{ij}\in \C$.
Hence 
\begin{equation}\label{eq:1}
\u^{mid}=\sum_{e_{ij}\in K(U)}x_{ij}a_{ij}\v_{ij}^{k-2}=\sum_{e_{ij}\in K(U)}y_{ij}\v_{ij}^{k-2}.\end{equation}
Since $k>n-1$, we may apply Lemma~\ref{lemma:vanishing} to $K_U$ to obtain $s_{k-2}(K_U)=0$, which means the vectors $\{\v_{ij}^{k-2}\big| e_{ij}\in K(U)\}$ are linearly independent. So it follows from \eqref{eq:1} that $x_{ij}a_{ij}=y_{ij}$. So it follows from \eqref{eq:2} that 
\[\u=\sum_{e_{ij}\in K(U)}x_{ij}\v_{ij}^k.\]
This is exactly what we need and the induction is complete.
\end{enumerate}
\end{proof}

We need the following technical lemma in the proof of Lemma~\ref{lemma:disconnect edge}.
\begin{lemma}\label{lemma:technical}
If $1\leq p\leq n, 1\leq q\leq n$, then 
\[\dfrac{p(p-1)}{2}+\dfrac{q(q-1)}{2}+n< \dfrac{(p+q)(n+1)}{2}.\]
\end{lemma}
\begin{proof}
\begin{equation*}
\begin{array}{cl}
&\dfrac{p(p-1)}{2}+\dfrac{q(q-1)}{2}+n-\dfrac{(p+q)(n+1)}{2}\\[0.5ex]
=&\dfrac{1}{2}\left(p^2-(n+2)p+q^2-(n+2)q\right)+n\\[1.5ex]
=&\dfrac{1}{2}\left((p-\dfrac{n+2}{2})^2+(q-\dfrac{n+2}{2})^2\right)-(\dfrac{n+2}{2})^2+n\\[1.0ex]
\leq & (\dfrac{n}{2})^2-(\dfrac{n+2}{2})^2+n\\[1.0ex]
=&-1 <0.
\end{array}
\end{equation*}
\end{proof}

We now study the effect that disconnecting the graph by deleting edges has on the statistic $s_k$. We call the following lemma the {\it Disconnecting Lemma}. It will be used to prove Theorem~\ref{theorem:edge connected}, and will also be used repeatedly in Section~\ref{sec: bound}.

\begin{lemma}\label{lemma:disconnect edge}
Given a connected graph $\Ga=(V, E)$, if removing $n$ edges $e_{i_1j_1}, e_{i_2j_2}$, ..., $e_{i_{n}j_{n}}$ leaves the graph with two connected components $\Ga_1=(V_1, E_1)$ and $\Ga_2=(V_2, E_2)$, but removing any $(n-1)$ of them will not disconnect the graph, then 
\[s_{k}(\Ga)=s_{k}(\Ga_1)+s_{k}(\Ga_2)\]
for any $k\geq n-1$.
\end{lemma}
\begin{proof}
For any $1\leq t\leq n$, either $v_{i_t}\in V_1, v_{j_t}\in V_2$ or $v_{i_t}\in V_2, v_{j_t}\in V_1$.  Without loss of generality, we may assume $v_{i_t}\in V_1$ for all $1\leq t\leq n$. Let $V_3=\{v_{i_t}\big| 1\leq t\leq n\}$, $V_4=\{v_{j_t}\big| 1\leq t\leq n\}$, $E_5=\{e_{i_1j_1}, e_{i_2j_2}, ..., e_{i_{n}j_{n}}\}.$ Note that even though $e_{i_sj_s}$ and $e_{i_tj_t}$ are assumed to be distinct for $s\neq t$, it could still happen that $v_{i_s}=v_{i_t}$ or $v_{j_s}=v_{j_t}$. So we have $|V_{3}|\leq n, |V_{4}|\leq n$ and $|E_5|=n$.  Define $\Ga'=(V', E')$ by $V'=V_3\cup V_4$, $E'=K(V_3)\cup K(V_4)\cup E_5$. 

We claim the graph $\Ga'$ has at leat one vertex of degree less than $n+1$. Otherwise, $\Ga'$ should have at least $\dfrac{(|V_3|+|V_4|)(n+1)}{2}$ edges. By Lemma~\ref{lemma:technical}, this is greater than $\dfrac{|V_{3}|(|V_3|+1)}{2}+\dfrac{|V_4|(|V_4|+1)}{2}+n$, which is the actually number of edges of $\Ga'$, contradction.

Assume without loss of generality that $v_{i_1}$ is of degree less than $n+1$.  Define $\Ga''=(V'\backslash\{v_{i_1}\}, E'\backslash E_{v_{i_1}})$, where $E_{v_{i_1}}=\{e\in E'\big| e \mbox{\ contains\ } v_{i_1}\}$. Then we may apply the Deleting Lemma~\ref{lemma:deleting} to $v_{i_1}$ and $\Ga'$ to get $s_{k}(\Ga')=s_k(\Ga'')$.

By the same argument we can show $\Ga''$ also has at least one vertex of degree less than $n+1$. Repeatedly using the Deleting Lemma, we finally conclude $s_{k}(\Ga')=0$.

Now assume there is a linear relation among $\{\v_{ij}^{k}\big| e_{ij}\in E\}$:
\begin{equation}\label{eq:3}\sum_{e_{ij}\in E} u_{ij}\v_{ij}^{k}=\0.\end{equation}
We may split the left hand side and rewrite it as
\begin{equation}\label{eq:4}
\sum_{e_{ij}\in E_1}u_{ij}\v_{ij}^k+\sum_{e_{ij}\in E_2}u_{ij}\v_{ij}^k+\sum_{e_{ij}\in E_5}u_{ij}\v_{ij}^k=\0.\end{equation}
As both $\ds{\sum_{e_{ij}\in E_2}u_{ij}\v_{ij}^k}$ and $\ds{\sum_{e_{ij}\in E_5}u_{ij}\v_{ij}^k}$ vanishes on $V_1\backslash V_3$, so does $\ds{\sum_{e_{ij}\in E_1}u_{ij}\v_{ij}^k}$. Then we may
apply Lemma~\ref{lemma:edge vector self contained} to $\Ga_1$ and $\ds{\sum_{e_{ij}\in E_1}u_{ij}\v_{ij}^k}$ to get 
\[\sum_{e_{ij}\in E_1}u_{ij}\v_{ij}^k=\sum_{e_{ij}\in K(V_{3})}m_{ij}\v_{ij}^k,\]
for some $m_{ij}\in \C$.
Similarly, we have
\[\sum_{e_{ij}\in E_2}u_{ij}\v_{ij}^k=\sum_{e_{ij}\in K(V_{4})}n_{ij}\v_{ij}^k,\]
for some $n_{ij}\in \C$.
So \eqref{eq:4} becomes
\[\sum_{e_{ij}\in K(V_{3})}m_{ij}\v_{ij}^k+\sum_{e_{ij}\in K(V_{4})}n_{ij}\v_{ij}^k+\sum_{e_{ij}\in E_5}u_{ij}\v_{ij}^k=\0.\]
Because $s_{k}(\Ga')=0$, the above equation forces $m_{ij}=0, \forall e_{ij}\in K(V_3)$; $\ n_{ij}=0, \forall e_{ij}\in K(V_4);\ $ $u_{ij}=0, \forall e_{ij}\in E_5$.
So \eqref{eq:4}, hence \eqref{eq:3}, is in fact the sum of two linear relations:
\[\sum_{e_{ij}\in E_1}u_{ij}\v_{ij}^k=\0, \ \mbox{and}\]
\[\sum_{e_{ij}\in E_2}u_{ij}\v_{ij}^k=\0.\]
So $s_{k}(\Ga)=s_{k}(\Ga_1)+s_{k}(\Ga_2)$.
\end{proof}

Now we are ready to prove Theorem~\ref{theorem:edge connected}.
\begin{proof}[Proof of Theorem~\ref{theorem:edge connected}]
Proof by contradiction. Assume to the contrary that $\Ga$ is not $d$-edge-connected, then there exists $n$ edges $e_{i_1j_1}, e_{i_2j_2}, ..., e_{i_nj_n}$ with $n<d$, such that upon removing these edges, $\Ga$ becomes disconnected. Also we may assume $n$ is the smallest number with such property. 
 
  Assume upon removing these edges, $\Ga$ becomes the disjoint union of two connected components $\Ga_1=(V_1, E_1)$ and $\Ga_2=(V_2, E_2)$.  Then by Lemma~\ref{lemma:disconnect edge}, $s_{d-2}(\Ga)=s_{d-2}(\Ga_1)+s_{d-2}(\Ga_2)$. Repeated application of the Deleting Lemma to $\Ga_1$ will yield $s_{d-2}(\Ga_1)=0$. Similarly, $s_{d-2}(\Ga_2)=0$. So $s_{d-2}(\Ga)=0$.  
  
  And we always have $s_{d-1}(\Ga)=s_{d}(\Ga)=0$ by Lemma~\ref{lemma:vanishing}. So $c_{d}=s_{d}+s_{d-2}-2s_{d-1}=0$. This gives a contradiction.   
\end{proof}

Theorem~\ref{theorem:vertex connected} will be proved in a similar fashion.
\begin{proof}[Proof of Theorem~\ref{theorem:vertex connected}]
Proof by contradiction. Assume to the contrary that $\Ga$ is not $(\lceil\dfrac{d}{2}\rceil+1)$-vertex connected, then there exists $n\leq \lceil\dfrac{d}{2}\rceil$ (and we may assume $n$ is the smallest integer with the following property), such that there are $n$ vertices, which we may assume without loss of generality are $U=\{v_1, v_2, ..., v_n\}$,  such that the graph $\Ga'=(V', E')$ defined by 
\[V'=V\backslash U,\ E'=E\backslash E_U,\]
is disconnected, where $E_{U}=\{e_{ij}\in E\big| v_{i}\in U \mbox{\ or\ } v_{j}\in U\}$. 

Becase of the minimality of $n$, we know $\Ga'$ consists of exactly two connected components, denote them by $\Ga_1=(V_1, E_1)$ and $\Ga_2=(V_2, E_2)$. For each vertex $v_{i}$ in $U$, the edge set $E_{v_i}$ decomposes into three parts:
\[E_{v_{i}}=E_{v_i}^1\sqcup\ E_{v_i}^2\sqcup\ E_{v_i}^U,\]
where $E_{v_i}^1=\{e_{ij}\in E\big| v_{j}\in V_1\}, E_{v_i}^2=\{e_{ij}\in E\big| v_{j}\in V_2\}, E_{v_{i}}^U=\{e_{ij}\in E\big| v_{j}\in U\}$.
If $E_{v_i}^1=\emptyset$, then removing $U\backslash\{v_i\}$ already disconnects $\Ga$, thus contradicts the minimality of $n$. So $E_{v_i}^1\neq \emptyset$. For same reason, $E_{v_i}^2\neq \emptyset$. 

Let $\ds{E_3=\bigcup_{v_i\in U}E_{v_i}^1}$, $\ds{E_4=\bigcup_{v_i\in U}E_{v_i}^2}$ and $\ds{E_{5}=\bigcup_{v_{i}\in U}E_{v_i}^U=E\cap K(U)}$. A linear relation among $\{\v_{ij}^{d-2}\big| e_{ij}\in E\}$ 
\begin{equation}\label{eq:5}
\sum_{e_{ij}\in E}u_{ij}\v_{ij}^{d-2}=\0
\end{equation}
may be rewritten as 
\begin{equation}\label{eq:6}
\sum_{e_{ij}\in E_1\cup E_3}u_{ij}\v_{ij}^{d-2}+\sum_{e_{ij}\in E_2\cup E_4}u_{ij}\v_{ij}^{d-2}+\sum_{e_{ij}\in E_5}u_{ij}\v_{ij}^{d-2}=\0.
\end{equation}
Since $|E_{v_1}^1|+|E_{v_1}^2|+|E_{v_1}^U|=d$, so either $|E_{v_1}^1|\leq \dfrac{d}{2}$ or $|E_{v_1}^2|\leq \dfrac{d}{2}$. We may assume without loss of generality that $|E_{v_1}^1|\leq \dfrac{d}{2}$. Since both $\ds{\sum_{e_{ij}\in E_2\cup E_4}u_{ij}\v_{ij}^{d-2}}$ and $\ds{\sum_{e_{ij}\in E_5}u_{ij}\v_{ij}^{d-2}}$ vanishes on $V_1$, it follows from \eqref{eq:6} that $\ds{\sum_{e_{ij}\in E_1\cup E_3}u_{ij}\v_{ij}^{d-2}}$ also vanishes on $V_1$.  Define $\Ga_3=(V_1\cup U, E_1\cup E_3\cup K(U))$ and apply Lemma~\ref{lemma:edge vector self contained} to $\Ga_3$ and $\ds{\sum_{e_{ij}\in E_1\cup E_3}u_{ij}\v_{ij}^{d-2}}$, we have
\begin{equation}\label{eq:7}
\sum_{e_{ij}\in E_1\cup E_3}u_{ij}\v_{ij}^{d-2}=\sum_{e_{ij}\in K(U)}x_{ij}\v_{ij}^{d-2},\end{equation}
for some $x_{ij}\in \C$.

\noindent {\bf Claim:} $s_{d-2}(\Ga_3)=0.$

\noindent {\bf Proof of the claim: }
The degree of $v_1$ inside $\Ga_3$ is $\leq \dfrac{d}{2}+\lceil\dfrac{d}{2}\rceil-1\leq \dfrac{d}{2}+\dfrac{d}{2}+\dfrac{1}{2}-1<d$. So we may apply the Deleting Lemma to $\Ga_3$ and $v_1$ to remove $v_1$ and edges containing $v_1$ from $\Ga_3$ to form a new graph, while keeping $s_{d-2}$ unchanged. If there is a vertex in the new graph of degree less than $d$, we may repeat this to form another new graph. Keep this process as long as there is a vertex in the new graph of degree less than $d$. This process will only stop when all the vertices are removed. This is because there is no subgraph (other than the empty one) of $\Ga_3$ whose vertices are all of degree greater than or equal to $d$. 

Assume there is a such one, which we may call $\Ga_6=(V_6, E_6)$ with $V_6\subseteq V_1\cup U$, $E_6\subseteq E_1\cup E_3\cup K(U)$. First we show $V_6\cap V_1=\emptyset$. If not, assume $v_{t}\in V_6\cap V_1$. Since $E_{v_1}^1\neq \emptyset$, we may assume $e_{1s}\in E_{v_1}^1$, then $v_s\in V_1$. As $\Ga_1$ is connected, there is a path from $v_t$ to $v_{s}$ in $\Ga_1$. Let's say it is $e_{ti_1}=e_{i_0i_1}, e_{i_1i_2}, ..., e_{i_{p-1}i_p}, e_{i_pi_{p+1}}=e_{i_ps}$. Let $u=\max\{j\big| 0\leq j\leq p+1, v_{j}\in V_6\}$. But then we find there is no way for $v_u$ to have degree greater than or equal to $d$ in $\Ga_6$.  So $V_6\subseteq U$. But $|U|\leq \lceil\dfrac{d}{2}\rceil$, this contradicts the assumption that every vertex in $\Ga_6$ has degree greater than or equal to $d$. 

Therefore $s_{d-2}(\Ga_3)=0$ and we have proved the claim.
\vskip 5mm

It follows from the above claim that equation \eqref{eq:7} forces $u_{ij}=0$ for $e_{ij}\in E_1\cup E_3$. Then it follows from \eqref{eq:6} that 
\begin{equation}\label{eq:8}
\sum_{e_{ij}\in E_2\cup E_4}u_{ij}\v_{ij}^{d-2}+\sum_{e_{ij}\in E_5}u_{ij}\v_{ij}^{d-2}=\0.
\end{equation}
Define $\Ga_4=(V_2\cup U, E_2\cup E_4\cup E_5)$. Applying the Deleting Lemma to $\Ga_4$ gives $s_{d-2}(\Ga_4)=0$. So \eqref{eq:8} forces $u_{ij}=0$ for all $e_{ij}\in E_2\cup E_4\cup E_5$. 

So \eqref{eq:6}, hence \eqref{eq:5} is a trivial linear relation, which means $s_{d-2}(\Ga)=0$.  So $c_{d}(\Ga)=s_{d}(\Ga)+s_{d-2}(\Ga)-2s_{d-1}(\Ga)=0$, a contradiciton.
\end{proof}

\section{\bf An upper bound for the second Betti number of compact Hamiltonian GKM manifolds}\label{sec: bound}
We now turn to the geometric consequences.
\begin{theorem}\label{thm:bound}
Given $\Ga=(V,E)$ a connected regular graph of degree $d\geq 2$, if $c_{d}(\Ga)=1$, then $c_{d-1}(\Ga)\leq \dfrac{m-2}{d-1}$.  As a consequence, if $M$ is a $2d$-dimensional connected compact Hamiltonian GKM manifold whose moment map is in general position, we have $\beta_{2}(M)=\beta_{2d-2}(M)\leq \dfrac{m-2}{d-1}$, where $\beta_{i}(M)$ is the $i$-th geometric Betti number of $M$, $m$ is the sum of all the Betti numbers, which is equal to the number of vertices in the GKM graph of the manifold.
\end{theorem}

\begin{corollary}\label{corollary:increasing}
If $M$ is a $8$ or $10$-dimensional connected compact Hamiltonian GKM manifold whose moment map is in general position, then $M$ has nondecreasing even Betti numbers up to half dimension: $\beta_{0}(M)\leq \beta_{2}(M)\leq \beta_{4}(M)$.
\end{corollary}
\begin{proof}
This is a straightforward calculation using Theorem~\ref{thm:bound} and Poincare duality. 

In the case of $8$-dimensional manifold, it follows from $\beta_2(M)\leq \dfrac{m-2}{3}$ that $\beta_4(M)=m-2-2\beta_2(M)\geq m-2-\dfrac{2(m-2)}{3}=\dfrac{m-2}{3}\geq \beta_2(M)$. 

In the case of $10$-dimensional manifold, it follows from $\beta_2(M)\leq \dfrac{m-2}{4}$ that $\beta_{4}(M)=\dfrac{1}{2}(m-2-2\beta_2(M))\geq \dfrac{m-2}{4}\geq \beta_2(M)$.

In both cases, we have $\beta_{0}(M)=1\leq \beta_{2}(M)$ since the symplectic form represents a nontrivial cohomology class.
\end{proof}

\begin{remark}
Corollary~\ref{corollary:increasing} shows the answer is positive for Question~\ref{question:tolman} by Tolman, in the case of $8$ and $10$ dimensional Hamiltonian GKM manifolds whose moment map is in general position.
\end{remark}

\begin{definition}
A graph $\Ga=(V, E)$ is called $k$-trimmed for some positive integer $k$, if 
\begin{itemize}
\item each vertex of $\Ga$ is of degree at least $k+1$,
\item each connected component of $\Ga$ is $(k+1)$-edge-connected.
\end{itemize}
\end{definition}

\begin{lemma}
Every graph $\Ga=(V, E)$ has a unique maximal $k$-trimmed subgraph, which we will denote by $\tilde{\Ga}^{k}$. It can be obtained by the following algorithm:
\begin{enumerate}
\item If $\Ga$ is $k$-trimmed, stop.
\item If $\Ga$ contains a vertex $v_i$ of degree less than or equal to $k$, we define $\Ga_1=(V_1, E_1)$ by $V_1=V\backslash\{v_i\}$ and $E_1=E\backslash E_{v_i}$, where $E_{v_i}$ is the set of edges containing $v_i$.
\item If every vertex of $\Ga$ is of degree $k+1$ or higher, and there exists $E'=\{e_{i_1j_1}, ..., e_{i_tj_t}\}\subseteq E$, such that $t\leq k$ and removing these edges would increase the number of connected components of $\Ga$, but removing any $t-1$ among them would not, then we define $\Ga_1=(V_1, E_1)$ by $V_1=V$, and $E_1=E\backslash E'$.
\item Repeat the above steps to $\Ga_1$.
\end{enumerate}
This process will finally stop. The resulting graph, which could be empty,  is $\tilde{\Ga}^{k}$.

By the Deleting Lemma~\ref{lemma:deleting} and the Disconnecting Lemma~\ref{lemma:disconnect edge}, we  see that $s_{k-1}(\tilde{\Ga}^{k})=s_{k-1}(\Ga)$.
\end{lemma}
\begin{proof}
The union of two $k$-trimmed subgraphs of $\Ga$ is also $k$-trimmed, so there is a unique maximal $k$-trimmed subgraph.

If $\Ga$ is itself $k$-trimmed, the lemma is trivial. Otherwise, following the algorithm, we get a sequence of graphs $\Ga_1, \Ga_2, ..., \Ga_p$, where $\Ga_{i+1}$ is a subgraph of $\Ga_i$ obtained from $\Ga_i$ either as in Case (2) or as in Case (3) stated in the lemma, $\Ga_p$ is $k$-trimmed. This sequence is not necessarily unique, but we will show in any case $\Ga_p=\tilde{\Ga}^k$. 

Assume $\tilde{\Ga}^k=(\tilde{V}^k, \tilde{E}^k)$. First, $\Ga_p$ is a subgraph of $\tilde{\Ga}^k$ since $\tilde{\Ga}^k$ is the unique maximal $k$-trimmed subgraph of $\Ga$. Secondly, if $\Ga$ is as in Case (2), then $v_i\notin \tilde{V}^k$, so $\tilde{\Ga}^k$ is a subgraph of $\Ga_1$. If $\Ga$ is as in Case (3), then $E'\cap \tilde{E}^k=\emptyset$, so $\tilde{\Ga}^k$ is also a subgraph of $\Ga_1$. Then we can show inductively that $\tilde{\Ga}^k$ is a subgraph of $\Ga_p$.   So $\Ga_p=\tilde{\Ga}^k$.
\end{proof}

We make the following definition so we can make the statements and proofs in the rest of the section more concise. 
\begin{definition}
For any $d\geq 2$, we call a graph $\Ga=(V, E)$ is {\it of type $A_{d}$} if 
\begin{itemize}
\item each vertex of $\Ga$ is of degree $d$ or $d-1$;
\item each connected component of $\Ga$ has at least one vertex of degree $d-1$;
\item each connected component of $\Ga$ is $(d-1)$-edge-connected.
\end{itemize}
\end{definition}

\begin{proposition}\label{prop:bound for s_d-3}
Assume $\Ga=(V, E)$ is a graph of type $A_{d}$,  then 
\begin{equation}\label{eq:s_d-3}
s_{d-3}(\Ga)\leq \dfrac{n_{d}(\Ga)}{d-1}+\pi_{0}(\Ga),\end{equation}
where $n_{d}(\Ga)=\{v_i\in V\big| \lambda(v_i)=d\}$ is the number of vertices of degree $d$, $\pi_{0}(\Ga)$ is the number of connected components of $\Ga$.
\end{proposition}
\begin{proof}
We are going to use induction on the size of $|V|$. The graph of type $A_d$ with least number of vertices is the complete graph on $d$ vertices. In this case it follows from Proposition~\ref{prop:complete} that $s_{d-3}(\Ga)=1$. And we obviously have $\dfrac{n_{d}(\Ga)}{d-1}+\pi_{0}(\Ga)=1$. So \eqref{eq:s_d-3} holds.

Now we consider $\Ga=(V, E)$ with $|V|=m> d$ and assume \eqref{eq:s_d-3} holds for $|V|<m$. If $\pi_{0}(\Ga)>1$, then each connected component of $\Ga$ is still of type $A_d$, and is of less vertices. So the induction hypothesis implies \eqref{eq:s_d-3} holds for each connected component. We may add them up to show \eqref{eq:s_d-3} holds for $\Ga$ as well.  Now we assume $\Ga$ is connected. 

Since $\Ga$ is of type $A_d$, we may pick a vertex $v_t$ of degree $d-1$. Define $\Ga'=(V', E')$ by $V'=V\backslash\{v_t\}$ and $E'=E\backslash E_{v_t}$, where $E_{v_t}$ is the set of edges containing $v_t$. By the corollary of the Deleting Lemma, we have $s_{d-3}(\Ga)\leq s_{d-3}(\Ga')+1$. Let $\tilde{\Ga'}^{d-2}$ be the maximal $(d-2)$-trimmed subgraph of $\Ga'$, then we have $s_{d-3}(\tilde{\Ga'}^{d-2})=s_{d-3}(\Ga')$. Assume $\pi_{0}(\tilde{\Ga'}^{d-2})=p$,  $\tilde{\Ga'}^{d-2}=\ds{\bigsqcup_{i=1}^{p}\Ga_{i}}$, and $\Ga_i=(V_i, E_i)$.

Since $\Ga$ is $(d-1)$-edge-connected, for each $\Ga_i$ there must be at  least $d-1$ vertices in $V_i$ whose degree in $\Ga$ were $d$ but now has degree $d-1$ in $\Ga_i$. So \[\ds{\sum_{i=1}^{p}}n_{d}(\Ga_i)\leq n_{d}(\Ga)-(d-1)p.\]

Since $\tilde{\Ga'}^{d-2}$ is also of type $A_d$ and with less vertices than $\Ga$, by the induction hypothesis we have
\[s_{d-3}(\tilde{\Ga'}^{d-2})\leq \dfrac{n_{d}(\tilde{\Ga'}^{d-2})}{d-1}+\pi_{0}(\tilde{\Ga'}^{d-2}).\]
So 
\[
\begin{array}{cl}
s_{d-3}(\Ga)&\leq s_{d-3}(\tilde{\Ga'}^{d-2})+1\\[0.5ex]
&\leq \dfrac{n_{d}(\tilde{\Ga'}^{d-2})}{d-1}+\pi_{0}(\tilde{\Ga'}^{d-2})+1\\[1.2ex]
&=\dfrac{1}{d-1}\ds{\sum_{i=1}^p n_{d}(\Ga_i)} + p+1\\
&\leq \dfrac{1}{d-1}(n_{d}(\Ga)-(d-1)p)+p+1\\[1.2ex]
&= \dfrac{n_{d}(\Ga)}{d-1}+1\\[0.5ex]
&=\dfrac{n_{d}(\Ga)}{d-1}+\pi_{0}(\Ga).
\end{array}
\]
This completes the induction step.
\end{proof}

Now Theorem~\ref{thm:bound} follows easily.
\begin{proof}[Proof of Theorem~\ref{thm:bound}] 
Since $c_{d}(\Ga)=1$, by Theorem~\ref{theorem:edge connected} we know that $\Ga$ is $d$-edge-connected.
Pick any edge $e_{ij}\in E$ and from a new graph $\Ga'=(V, E\backslash\{e_{ij}\})$.  Then $\Ga'$ is of type $A_d$. By Proposition~\ref{prop:bound for s_d-3}, we have 
\[s_{d-3}(\Ga')\leq \dfrac{m-2}{d-1}+1.\]
So $s_{d-3}(\Ga)\leq s_{d-3}(\Ga')+1\leq \dfrac{m-2}{d-1}+2.$ Hence 
\[c_{d-1}(\Ga)=s_{d-1}(\Ga)+s_{d-3}(\Ga)-2s_{d-2}(\Ga)\leq 0+\dfrac{m-2}{d-1}+2-2=\dfrac{m-2}{d-1},\]
where we have used the facts $s_{d-1}(\Ga)=0$ by the Deleting Lemma, and $s_{d-2}(\Ga)=c_{d}(\Ga)-(s_{d}(\Ga)-2s_{d-1}(\Ga))=1$.
\end{proof}

\end{document}